\documentclass[10pt,reqno]{amsart}
\usepackage{epsfig,amssymb,amsmath,version}
\usepackage{amssymb,graphicx,fancybox,mathrsfs,multirow}

\usepackage{url,hyperref}

\usepackage{subfigure}
\usepackage{color}
\usepackage{cases}
\usepackage{mathtools}

\catcode`\@=11 \theoremstyle{plain}
\@addtoreset{equation}{section}   
\renewcommand{\theequation}{\arabic{section}.\arabic{equation}}
\@addtoreset{figure}{section}
\renewcommand\thefigure{\thesection.\@arabic\c@figure}
\renewcommand{\thefigure}{\arabic{section}.\arabic{figure}}

\newtheorem{thm}{\bf Theorem}[section]
\newtheorem{lmm}{\bf Lemma}[section]
\newtheorem{cor}{\bf Corollary}[section]

\newtheorem{rem}{\bf Remark}[section]

\theoremstyle{definition}
\newtheorem{defn}{\bf Definition}[section]

\numberwithin{table}{section}

\newcommand{\D}[3]{{{}_{#1}}{\rm D}_{#2}^{#3}}
\newcommand{\I}[3]{{}_{#1}{\rm I}_{#2}^{#3}}
\newcommand{\CD}[3]{{}_{#1}^{C}{\rm D}_{#2}^{#3}}

\begin{document}
\bibliographystyle{plain}
\graphicspath{{./figs/}}
\baselineskip 13pt

{\title[Log orthogonal functions] {Log orthogonal functions:  approximation  properties and applications}
\author{Sheng Chen${}^{1}$\; and \;  Jie Shen${}^{2}$}
\thanks{${}^{1}$ Beijing Computational Science Research Center, Beijing 100193, P.R. China. \& School of Mathematics and Statistics, Jiangsu Normal University, Xuzhou 221116, China. The research of S. Chen  is partially supported by the Postdoctoral Science Foundation of China (Grant No. BX20180032, 2019M650459), NSFC 11801235,  and the Natural Science Foundation of the Jiangsu Higher Education Institutions of China (Grant No. BK20181002).
Email: shengchen@csrc.ac.cn.}
\thanks{${}^2$ Corresponding~ author. Department of Mathematics, Purdue University, West Lafayette,
  IN 47907-1957, USA. J.S. is  partially  supported by NSF grants
        DMS-1620262, DMS-1720442   and AFOSR grant FA9550-16-1-0102. Email: shen7@purdue.edu.}}

\begin{abstract}
We present  two new classes of orthogonal functions, log orthogonal functions (LOFs) and generalized log orthogonal functions (GLOFs), which are constructed  by applying a $\log$ mapping to  Laguerre polynomials. We develop basic approximation theory for these new orthogonal functions and apply them to solve several typical fractional differential equations whose solutions exhibit weak singularities. Our error analysis and numerical results show that our methods based on the  new orthogonal functions are particularly suitable for functions which have weak singularities at one endpoint, and can lead to exponential convergence rate, as opposed to  low algebraic rates if usual orthogonal polynomials are used.
\end{abstract}


 \maketitle
\section{Introduction}

It is well-known that classical
spectral methods can provide high accuracy for problems with smooth solutions \cite{Sp1,Boyd01,Sp2,ShenTangWang2011}, but may not have any advantage for problems with non-smooth solutions. On the other hand, solutions of many practical applications involve  weakly singular solutions, such as in non-smooth domains, with non-matching boundary conditions,  in integral equations with singular/weakly singular  kernels, and in fractional differential equations. One effective strategy in finite differences/finite elements is to employ a local adaptive procedure  \cite{Morin1}, but this strategy can  not be effectively extended to the global spectral method. Hence, in order to develop accurate spectral methods for problems with non-smooth solutions, one has to choose suitable basis functions
which can effectively approximate the underlying non-smooth solutions.
A popular strategy in this regard is to enrich the usual polynomial based approximation space by adding special functions which capture the singular behavior of the underlying problem, for examples, the
so called singular functions method \cite{MR0443377},
extended or generalized finite element method
(GFEM/XFEM) (cf. \cite{Babuvska2012,Fires} and the references therein), and in  the context of spectral methods, the enriched spectral-tau method \cite{Boyd1991} and the enriched spectral-Galerkin method \cite{Che.S18}. Another effective  strategy in  the context of   spectral method is to construct special orthogonal functions which are suitable for a certain  class of problems with singular behaviors.
In addition to classical orthogonal polynomials, one can use suitable mappings to classical orthogonal polynomials to construct orthogonal functions  in weighted Sobolev spaces, leading to the so called
  mapped spectral methods (cf. \cite{Boyd01} for an extensive disscussion). In  \cite{MR932149,MR892255,She.W04,She.W05b}, the mapped spectral methods have been successfully used in constructing efficient spectral methods for problems in unbounded domains, and in M\"untz Galerkin method \cite{MR3531734} for a special class of singular problems. On the other hand,
  Boyd \cite{Boyd86} briefly discussed several possible alternatives  to deal with weak singularities at both endpoints through different mappings.

 In this paper, we are concerned with problems which exhibit weakly singular behaviors at the initial time for initial value problems or at one endpoint for boundary value problems. We  construct special classes of orthogonal functions, through a suitable log mapping to Laguerre functions,  which  are capable of resolving  weak singularities. We shall develop basic approximation results for  two new classes of  orthogonal functions, log orthogonal functions (LOFs) and generalized log orthogonal functions (GLOFs). In particular, these results indicate that approximation by the LOFs and GLOFs to functions behaving like $t^r(-\log t)^{k}$ near $t=0$ will converge exponentially for any real $r\geq 0,~k\in\mathbb{N}_0$. In fact, we believe that this is the first set of basis functions which can approximate regular polynomials $t^k\,(k\in \mathbb{N}_0)$ and  weakly singular functions like $t^r(-\log t)^{k}~(0<r<1,~k\in\mathbb{N}_0)$ with exponential accuracy.  Thus, LOFs and GLOFs are particularly suitable for problems whose solutions exhibit weak singularities behaving like $\sum_{i} c_it^{r_i}(-\log t)^{k_i}$ near $t=0$ for small $r_i>0$. 
 In particular, solutions of  time fractional differential equations  and boundary value problems with one-sided fractional derivatives fall into this category.   Hence, the spectral methods using GLOFs that we propose in this paper can be used to deal with
a large class of fractional differential equations having weak singularities at the initial time or at one endpoint.
 
{ Numerical solution of fractional differential equations (FDEs) has been a subject of intensive investigation in recent years, cf. for instance \cite{Mee.T04,Sun.W06,EHR07,Erv.R07,JLZ13} (and the references therein) for finite-difference and finite-element methods, and \cite{LZL12, Li.X09,Li.X10,zayernouri2013fractional,chen2016generalized} (and the references therein) for spectral methods. However, most of the error analysis are derived in the context of  usual Sobolev spaces which are not quite suitable for FDEs. In our previous works \cite{chen2016generalized,mao2016efficient}, we developed an  error analysis using the generalized Jacobi functions based on the non-uniformly weighted spaces which showed that, for some model FDEs whose solutions behave as $(t-a)^r (b-t)^s h(t)$ with known $r,\,s>0$  and  smooth $h(t)$, the error may converge exponentially as long as the data function is smooth in the usual sense. However, for more general FDEs such as those with variable coefficients or nonlinearity, the singular behavior of their solutions is unknown {\it a priori}, so approximations by generalized Jacobi functions can not achieve desired accuracy.
 However, GLOFs can handle functions with unknown endpoint singularity since they can approximate  singular functions of the form $\sum_{i}c_it^{r_i}(-\log t)^{k_i}$ with  exponential accuracy. 
In fact, we show in Corollary \ref{singularity_appro} that GLOFs can achieve exponential convergence for typical singular functions $t^r(-\log t)^{k}$ for any  $r\geq 0,~k\in\mathbb{N}_0$. To the best of our knowledge, this is  the first set of basis functions which can approximate the singular solutions of the form in \eqref{singularity_solu} with exponential accuracy.}

The rest of the paper is organized as follows. In the next section, we introduce the LOFs,  derive  optimal projection and interpolation errors in weighted pseudo-derivatives which are adapted to the involved mapping. In Section 3, we introduce the GLOFs which involve an additional parameter so  are more flexible than LOFs, and derive the corresponding optimal projection and interpolation errors. 
In Section 4, we apply GLOFs to solve several typical classes of fractional differential equations, and derive optimal  error estimates which indicate, in particular, that for solutions and data functions having weak singularities at $t=0$ or one endpoint,  errors of the proposed GLOF-Galerkin methods will converge exponentially.
In each of the sections 2, 3 and 4, we also present numerical results to validate the theoretical estimates and to show the effectiveness of our new method. Some concluding remarks are given in the last section.

\section{Log orthogonal functions}
In this section, we introduce the log orthogonal functions, derive the corresponding approximation theory, and present numerical results to validate the theoretical estimates and to show their effectiveness.

To fix the idea, we consider the canonical  time interval $I=(0,1)$. Throughout the paper, we  use the mapping
\begin{equation}\label{mapping}
y(t):=-(\beta+1)\log{t}, \quad t\in I,
\end{equation}
 to map $I$ to $\mathbb{R}^+$.
We shall make use of  $\mathscr{L}^{(\alpha)}_n(y),~\alpha>-1$ which is the Laguerre polynomial of $y \in \mathbb{R}^+$, satisfying
\begin{equation}\label{GLagorth}
    \int_{0}^\infty \mathscr{L}_n^{(\alpha)}(y)\, \mathscr{L}_m^{(\alpha)}(y) \,y^\alpha e^{-y} \, {\rm d}y= \gamma _n^{(\alpha)} \delta_{mn},\quad \gamma _n^{(\alpha)}=\frac{\Gamma(n+\alpha+1)}{\Gamma(n+1)}.
\end{equation}
Some additional properties  of Laguerre polynomials  are listed in Appendix \ref{Appendix_Laguerre}, see also \cite{ShenTangWang2011,szego1975orthogonal}.

\subsection{Definition and properties}
\begin{defn}[LOFs]\label{def_LOF}{\em
Let $\alpha, \beta>-1$. We define the log  orthogonal functions by
\begin{equation}\label{log_basis}
\mathcal{S}^{(\alpha,\beta)}_n(t)=\mathscr{L}^{(\alpha)}_n(y(t))=\mathscr{L}^{(\alpha)}_n (-(\beta+1)\log{t}), \quad n=0,1,\ldots .
\end{equation}
}\end{defn}

From the properties of Laguerre polynomials listed in Appendix \ref{Appendix_Laguerre} and
 the following relations
 \begin{equation}\label{transform}
y=-(\beta+1)\log{t},\quad {\rm d}y=-(\beta+1)t^{-1}{\rm d} t,\quad \partial_t=-(\beta+1)t^{-1}\partial_y,
 \end{equation}
we can easily derive the  following lemma:
\begin{lmm}\label{lem:prop} The LOFs satisfy the following  properties:
\begin{itemize}
\item[\bf{P1}.]{\em Three-term recurrence relation:}
\begin{equation}\begin{aligned}\label{S_Three-term R}
&\mathcal{S}^{(\alpha,\beta)}_0(t)=1,\qquad \mathcal{S}_1^{(\alpha,\beta)}(t)=(\beta+1)\log{t}+\alpha+1,\\
&\mathcal{S}^{(\alpha,\beta)}_{n+1}(t)=\frac{2n+\alpha+1+(\beta+1)\log{t}}{n+1}\mathcal{S}^{(\alpha,\beta)}_n(t)-\frac{n+\alpha}{n+1}\mathcal{S}^{(\alpha,\beta)}_{n-1}(t).
\end{aligned}\end{equation}
\item[\bf{P2}.]{\em Derivative relations:}
\begin{equation}\label{S_Derivative}
(\beta+1)^{-1}{t}{\partial}_{t}\mathcal{S}_n^{(\alpha,\beta)}(t)= \mathcal{S}_{n-1}^{(\alpha+1,\beta)}(t)=\sum_{l=0}^{n-1}\mathcal{S}_{l}^{(\alpha,\beta)}(t), \quad n\geq 1.
\end{equation}

\item[\bf{P3}.]{\em Orthogonality:}
\begin{equation}\label{S_orth}
    \int_{0}^1 \mathcal{S}_n^{(\alpha,\beta)}(t)~  \mathcal{S}_m^{(\alpha,\beta)}(t)\,~ (-\log{t})^{\alpha} \,t^\beta \, {\rm d}t=\gamma^{(\alpha,\beta)}_n \delta_{mn},\end{equation}
    where $$ \gamma^{(\alpha,\beta)}_n:= \frac{\Gamma(n+\alpha+1)}{(\beta+1)^{\alpha+1}\Gamma(n+1)} .$$
\item[\bf{P4}.]{\em Sturm-Liouville problem:}
\begin{equation}\label{S_SLProb}
  {(-\log{t})^{-\alpha}}t^{-\beta}\partial_t\left( (-\log{t})^{\alpha+1}{t^{\beta+2}}\partial_t \mathcal{S}_n^{(\alpha,\beta)}(t)\right)+n(\beta+1) \mathcal{S}_n^{(\alpha,\beta)}(t)=0.
\end{equation}
\item[\bf{P5}.]{\em Gauss-LOFs quadrature:}\\
 Let $\{y_j^{(\alpha)},~\omega_j^{(\alpha)}\}_{j=0}^N$ be the  Gauss-node and -weight of $\mathscr{L}_{n+1}^{(\alpha)}(y)$. Denote
\begin{equation}\label{node_weight}
\big\{{t}_j^{(\alpha,\beta)}:=e^{-(\beta+1)^{-1}y_j^{(\alpha)}},\quad \chi^{(\alpha,\beta)}_j={(\beta+1)^{-\alpha-1}\omega^{(\alpha)}_j}\big\}_{j=0}^N.
\end{equation}
 Then,
   \begin{equation}\label{Gauss_Quad}
  \int_{0}^1 p(t) (-\log{t})^\alpha t^{\beta} {\rm d}t= \sum_{j=0}^N p({t}_j^{(\alpha,\beta)})\,{\chi}_j^{(\alpha,\beta)},\quad \forall \, p\in \mathcal{P}^{\log{t}}_{2N+1},
  \end{equation}
  where
  \begin{equation}\label{P_K^logt}
 \mathcal{P}^{\log{t}}_K:=span\{1,\log{t},~(\log{t})^2,\ldots,(\log{t})^K\}.
 \end{equation}
\end{itemize}
\end{lmm}
\begin{proof}
 The three-term recurrence relation \eqref{S_Three-term R} is a straightforward result from \eqref{Three-term R} with the variable transform  \eqref{transform}.

 \eqref{S_Derivative} can be obtained from the relations \eqref{LaguDerivative3} and \eqref{transform}. Indeed,
$$(\beta+1)^{-1}t\partial_t \{\mathcal{S}_n^{(\alpha,\beta)}(t)\}\overset{\eqref{transform}}{=}-\partial_y \mathscr{L}^{(\alpha)}_n(y)\overset{\eqref{LaguDerivative3}}{=}\mathscr{L}^{(\alpha+1)}_{n-1}(y)=\mathcal{S}_{n-1}^{(\alpha+1,\beta)}(t).$$

We derive from $y=-(\beta+1)\log{t}$ that
\begin{equation*}\begin{aligned}
\int_{0}^1 \mathcal{S}_n^{(\alpha,\beta)}(t)~\mathcal{S}_m^{(\alpha,\beta)}(t)\,(-\log{t})^{\alpha} t^{\beta}  \, {\rm d}t=\frac{1}{(\beta+1)^{\alpha+1}}\int_{0}^\infty \mathscr{L}_{n}^{(\alpha)}(y)~ \mathscr{L}_{m}^{(\alpha)}(y)\,y^{\alpha}e^{-y}  \, {\rm d}y.
\end{aligned}\end{equation*}
Hence, we have \eqref{S_orth}.

\eqref{S_SLProb} is valid since
 $$y=-(\beta+1)\log{t},\quad \partial_y=-(\beta+1)^{-1}t\,\partial_t,\quad \mathscr{L}_n^{(\alpha)}(y)=\mathcal{S}_n^{(\alpha,\beta)}(t)$$
 lead to
 $$
  y^{-\alpha}e^y\partial_y\left( y^{\alpha+1}e^{-y}\partial_y \mathscr{L}_n^{(\alpha)}(y)\right)=\frac{(-\log{t})^{-\alpha}t^{-\beta}}{\beta+1}\partial_t \left((-\log{t})^{\alpha+1}t^{\beta+2}\partial_t \mathcal{S}^{(\alpha,\beta)}_n\right).$$

Finally, setting $t=e^{-(\beta+1)^{-1}y}$, we can obtain \eqref{node_weight}
from the Laguerre-Gauss quadrature:
\begin{equation*}\begin{aligned} \int_{0}^1 p(t) (-\log{t})^\alpha t^\beta & {\rm d}t=\int_{\mathbb{R}^+} p(x(y)) \frac{y^\alpha e^{-y}}{(\beta+1)^{\alpha+1}} \,{\rm d}y\\&= \sum_{j=0}^N p(x(y_j^{(\alpha)}))\frac{\omega_j^{(\alpha)}}{(\beta+1)^{\alpha+1}}= \sum_{j=0}^N p({t}_j^{(\alpha,\beta)}){\chi}_j^{(\alpha,\beta)}.
\end{aligned}\end{equation*}
\end{proof}

\begin{rem}
We used two parameters $\alpha$ and $\beta$ to provide better flexibility, e.g., they allow us to effectively deal with problems with weight $(-\log t)^\alpha t^\beta$. In the applications considered in this paper, the $\log$ term does not appear so we can take $\alpha=0$. On the other hand,   taking $\beta=0$ offers good approximation properties for problems with weight $(-\log t)^\alpha$, which can not be well approximated by classical orthogonal polynomials.
 \end{rem}

\subsection{Projection estimate}
Let $\alpha,\beta>-1$, and $\chi^{\alpha,\beta}(t):=(-\log{t})^\alpha t^\beta$. For any
 $u\in L^2_{\chi^{\alpha,\beta}}({I})$, we denote
 $\pi_N^{\alpha,\beta} u$ the projection from $L^2_{\chi^{\alpha,\beta}}$ to $\mathcal{P}^{\log{t}}_N$
by
\begin{equation}\label{projection1}
(u-\pi_N^{\alpha,\beta} u, v)_{\chi^{\alpha,\beta}}=\int_{0}^1\{u-\pi_N^{\alpha,\beta} u\}(t) \,v(t) \,\chi^{\alpha,\beta}(t) {\rm d}t=0, \quad \forall v\in \mathcal{P}^{\log{t}}_N.
\end{equation}
Thanks to the orthogonality of the basis $\{\mathcal{S}^{(\alpha,\beta)}_n\}_{n=0}^\infty$, we have
\begin{equation}\label{projection1}
\pi_N^{\alpha,\beta} u=\sum_{n=0}^N \hat{u}^{\alpha,\beta}_n \,\mathcal{S}^{(\alpha,\beta)}_n,\quad \hat{u}^{\alpha,\beta}_n=(\gamma^{(\alpha,\beta)}_n)^{-1}\int_{0}^1 u(t)\, \mathcal{S}^{(\alpha,\beta)}_n(t) \,\chi^{\alpha,\beta} (t) {\rm d} t.
\end{equation}

To better describe the projection error $\pi_N^{\alpha,\beta} u$, we  define  a pseudo-derivative
\begin{equation}\label{pseudo_deri}
\widehat{\partial}_t u:= t\partial_t u
\end{equation}
 and a non-uniformly weighted Sobolev  space
\begin{equation}\label{HilbertSpace}
{A}^{k}_{\alpha,\beta}({I}):=\{v\in L^{2}_{\chi^{\alpha,\beta}}({I}):~{\widehat{\partial}_{t}}^{\,j} v\in  L^{2}_{\chi^{\alpha+j,\beta}}({I}),~j=1,2,\ldots,k\},\quad k\in \mathbb{N},
\end{equation}
equipped with semi-norm and norm
\begin{equation*}
|v|_{{A}^{m}_{\alpha,\beta}}:=\|\widehat{\partial}_{t}^{\,m}{v}\|_{\chi^{\alpha+m,\beta}},\quad \|v\|_{{A}^{m}_{\alpha,\beta}}:=\big(\sum_{k=0}^m |v|^2_{{A}^{k}_{\alpha,\beta}}\big)^{{1}/{2}}.
\end{equation*}

\begin{thm}\label{thm_Proj}
Let $m,\, N,\,k\in \mathbb{N}$ and $\alpha,\beta>-1$. For any $u\in{A}^{m}_{\alpha,\beta}({I})$ and $0\leq k\leq \widetilde{m}:=\min\{m,N+1\}$, we have
\begin{equation}
\|\widehat{\partial}_{t}^{\,k} (u-\pi_N^{\alpha,\beta} u)\|_{\chi^{\alpha+k,\beta}}\leq  \sqrt{(\beta+1)^{k-\widetilde{m}}\frac{(N-\widetilde{m}+1)!}{(N-{k}+1)!}}~ \|\widehat{\partial}_{t}^{\,\widetilde{m}}
 u\|_{\chi^{\alpha+\widetilde{m},\beta}}.
\end{equation}
 In particular, fixing $\alpha=\beta=k=0$ and $m<N+1$, it holds that
\begin{equation}\label{eq_L2}
\|u-\pi_N u\|\leq  c N^{-\frac{m}{2}}~ \|\widehat{\partial}_{t}^{\,{m}}
 u\|_{\chi^{{m}}},
\end{equation}
where $\pi_N=\pi_N^{0,0}$ and $\chi^{{m}}=\chi^{{m},0}=(-\log{t})^m$ for notational simplicity.
\end{thm}
\begin{proof}
 For any $u\in{A}^{m}_{\alpha,\beta}({I})$, via   relations \eqref{S_Derivative} and \eqref{pseudo_deri}, we have
  \begin{equation}\label{rela_d}
 \widehat{\partial}_{t}^{\,l} \mathcal{S}^{(\alpha,\beta)}_n(t)=(\beta+1)^l\mathcal{S}^{(\alpha+l,\beta)}_{n-l}(t), \quad l\leq n.
 \end{equation}
 Then, it can be easily detected from the orthogonality {\bf{P3}} that
$$u(t)=\sum_{n=0}^\infty \hat{u}^{\alpha,\beta}_n \mathcal{S}_n^{(\alpha,\beta)}(t),\quad \|\widehat{\partial}_{t}^{\,l} u\|^2_{\chi^{\alpha+l,\beta}}=\sum_{n=l}^\infty (\beta+1)^{2l}\gamma^{(\alpha+l,\beta)}_{n-l} |\hat{u}^{\alpha,\beta}_n|^2,\quad l\geq 1.$$
Therefore,
\begin{equation*}\begin{aligned}
\|\widehat{\partial}_{t}^{\,k} (u-\pi_N^{\alpha,\beta} u)\|^2_{\chi^{\alpha+k,\beta}}&=\sum_{n=N+1}^\infty (\beta+1)^{2k}\gamma^{(\alpha+k,\beta)}_{n-k} |\hat{u}^{\alpha,\beta}_n|^2\\&\leq\max\{\frac{\gamma^{(\alpha+k,\beta)}_{n-k}}{\gamma^{(\alpha+\widetilde{m},\beta)}_{n-\widetilde{m}}}\}\sum_{n=N+1}^\infty (\beta+1)^{2k}\gamma^{(\alpha+\widetilde{m},\beta)}_{n-\widetilde{m}} |\hat{u}^{\alpha,\beta}_n|^2\\
&\leq(\beta+1)^{2(k-\widetilde{m})}\frac{\gamma^{(\alpha+k,\beta)}_{N+1-k}}{\gamma^{(\alpha+\widetilde{m},\beta)}_{N+1-\widetilde{m}}}\|\widehat{\partial}_{t}^{\,\widetilde{m}} u\|^2_{\chi^{\alpha+\widetilde{m},\beta}}
\\&\leq(\beta+1)^{k-\widetilde{m}}\frac{(N-\widetilde{m}+1)!}{(N-{k}+1)!}\|\widehat{\partial}_{t}^{\,\widetilde{m}} u\|^2_{\chi^{\alpha+\widetilde{m},\beta}}.
\end{aligned}\end{equation*}
Finally the proof of the special case can be proved by the following useful result:
 for any constant   $a, b\in {\mathbb R},$   $n\in {\mathbb N}, $  $n+a>1$ and $n+b>1$ (see  \cite[Lemma 2.1]{ZhWX13}),
\begin{equation}\label{Gammaratio}
\frac{\Gamma(n+a)}{\Gamma(n+b)}\le \nu_n^{a,b} n^{a-b},
\end{equation}
where
\begin{equation}\label{ConstUpsilon}
\nu_n^{a,b}=\exp\Big(\frac{a-b}{2(n+b-1)}+\frac{1}{12(n+a-1)}+\frac{(a-b)^2}{n}\Big).
\end{equation}
\end{proof}
{
\begin{rem}
 The essential difference between approximations by LOFs  and traditional polynomials can be explained by  the estimate \eqref{eq_L2}. 
In fact, since $\widehat{\partial}_t t^{r}=r t^{r}$, it's easy to check that $\|\widehat{\partial}_t^m t^r\|_{\chi^m}<\infty$ for all $r\geq 0$ and any positive integer $m$. So the LOFs can  approximate a function whose singularity behave as $\sum_{i} c_i t^{r_i}$ with exponential convergence.
   
  On the contrary, the polynomial approximation error depends on the regularity defined by the usual derivative. Specifically, for the classical polynomial projection $\Pi_N$: $L^2\rightarrow P_N^t:=span\{1,t,\ldots,t^N\}$, it holds that
 $$\|u-\Pi_Nu\|\leq c N^{-m} \|\partial_t^m u\|.$$
  Hence, functions behaving as $\sum_{i} c_i t^{r_i}$ with several small $r_i>0$ cannot be well approximated  by polynomials.
   \qed
 \end{rem}
 
 }
Then, we obtain \eqref{eq01} by taking $k=0,~\tilde{m}=m$ in Theorem \ref{thm_Proj}.

On the other hand,
\begin{equation*} \frac{r^2}{(\beta+1+2r)(\beta+1)}<1~\Longleftrightarrow~ {r^2}-{2(\beta+1)r-(\beta+1)^2}<0.
\end{equation*}
Finally, we take $k=0, m=N$ in Theorem \ref{thm_Proj} to obtain \eqref{eq1}.
 \end{proof}
\begin{rem}
 We observe from \eqref{eq1} that the projection error for $t^r$ converges exponentially fast if $(1-\sqrt{2})(\beta+1)<r<(1+\sqrt{2})(\beta+1)$. This range can be adjusted  by varying $\beta$.
However, we observe from \eqref{eq01} that, for    $r>(1+\sqrt{2})(\beta+1)$, we need $N>k_{r,\beta}$ to have any accuracy. \qed
\end{rem}

\end{proof}

\subsection{Interpolation estimate}
Let $ \{t^{(\alpha,\beta)}_j\}_{j=0}^N$ be the mapped Gauss points defined in \eqref{node_weight}.
We  define the mapped Lagrange functions
\begin{equation}\label{interpolation_2}
{l}_j\big(y(t)\big)=\frac{\prod\limits_{i\neq j} \big(y(t)-y(t^{(\alpha,\beta)}_i)\big)}{\prod\limits_{i\neq j}\big(y(t^{(\alpha,\beta)}_j)-y(t^{(\alpha,\beta)}_i)\big)}=\frac{\prod\limits_{i\neq j} \log (t^{(\alpha,\beta)}_i/t) }{\prod\limits_{i\neq j}\log (t^{(\alpha,\beta)}_i/t^{(\alpha,\beta)}_j)},
\end{equation}
 and the interpolation operator $\mathcal{I}^{\alpha,\beta}_N: C(I)\rightarrow P_N^{\log t}$ by
\begin{equation}\label{interpolation_1}
\mathcal{I}^{\alpha,\beta}_N v(t)=\sum_{j=0}^N {v(t^{(\alpha,\beta)}_j)} {l}_j\big({y}(t)\big).
\end{equation}
Obviously, we have $\mathcal{I}^{\alpha,\beta}_N v(t^{(\alpha,\beta)}_j)=v(t^{(\alpha,\beta)}_j)$, $j=0,1,\cdots,N$.

We first establish a  stability result.
\begin{thm}
For any $v\in C({I})\cap A^1_{\alpha,\beta}({I})$ and $\widehat{\partial}_t v\in L^{2}_{\chi^{\alpha,\beta}}({I})$, we have
\begin{equation}\label{stability}
\|\mathcal{I}^{\alpha,\beta}_N v\|_{\chi^{\alpha,\beta}}\leq c\sqrt{(\beta+1)^{\alpha}}\left(c^{\beta}_1{N^{-\frac{1}{2}}}\|\widehat{\partial}_t v\|_{\chi^{\alpha,\beta}}+c^{\beta}_2\sqrt{\log N}\|v\|_{A^1_{\alpha,\beta}}\right).
\end{equation}
where $c^{\beta}_1={(\beta+1)^{-\frac{1}{2}}},\quad c^{\beta}_2=2\sqrt{\max\{1,\beta+1\}}.$
\end{thm}
\begin{proof}
Let $t(y)=e^{-(\beta+1)^{-1}y}$ and  $\tilde{v}(y)=v(t(y))$. Via relations \eqref{interpolation_1} and \eqref{node_weight}, we have
\begin{equation*}
\mathcal{I}^{\alpha,\beta}_N v(t)=\mathcal{I}^{\alpha}_N \tilde{v}(y):=\sum_{j=0}^N {\tilde{v}(y^{(\alpha)}_j)} {l}_j({y}), \quad y\in\mathbb{R}^+.
\end{equation*}
Thanks to \cite[(3.12) with $\beta=1$]{ben2006generalized}, we have
$$\|\mathcal{I}^{\alpha}_N \tilde{v}\|_{y^{\alpha} e^{-x}}\leq c(N^{-\frac{1}{2}} ~\sqrt{M^{\tilde{v}}_1}+2\sqrt{\log{N}}~\sqrt{M^{\tilde{v}}_2}),$$
where
$$M^{\tilde{v}}_1=\int_{0}^\infty (\partial_y \tilde{v}(y))^2~y^{\alpha}e^{-y}~{\rm d} y,\quad M^{\tilde{v}}_2=\int_{0}^\infty  \left( \tilde{v}^2+y(\partial_y \tilde{v})^2\right)y^{\alpha}e^{-y}~{\rm d} y.$$
On the other hand, we have
\begin{equation*}\begin{aligned}
&\int_{0}^\infty (\partial_y \tilde{v}(y))^2~y^{\alpha}e^{-y}~{\rm d} y=\int_{0}^1 (\frac{t}{\beta+1}\partial_t {v}(t))^2~(-(\beta+1)\log{t})^{\alpha}t^{\beta+1}~\frac{\beta+1}{t}{\rm d} t\\
&=(\beta+1)^{\alpha-1}\int_{0}^1 (\widehat{\partial}_t {v}(t))^2~(-\log{t})^{\alpha}t^{\beta}{\rm d} t=(\beta+1)^{\alpha-1}\|\widehat{\partial}_t {v}\|^2_{\chi^{\alpha,\beta}}
\end{aligned}\end{equation*}
and
\begin{equation*}\begin{aligned}
\int_{0}^\infty  \left( \tilde{v}^2+y(\partial_y \tilde{v})^2\right)y^{\alpha}e^{-y}~{\rm d} y&=(\beta+1)^{\alpha+1}\int_{0}^1 \left(v^2+\frac{(-\log{t})}{\beta+1}(\widehat{\partial}_t v)^2\right)~(-\log{t})^{\alpha}t^{\beta}{\rm d} t\\
&\leq(\beta+1)^{\alpha}\max\{1,(\beta+1)\}\|{v}\|^2_{A^1_{\alpha,\beta}}.
\end{aligned}\end{equation*}
We can then  derive  \eqref{stability} by combing  the above relations.
\end{proof}

With the above stability result in hand, we can now establish an  estimate for the interpolation error.

\begin{thm}\label{thm_Interpolation}
Let $m$ and $N$ be positive integers, and  $\alpha,\beta>-1$. For any $v\in C({I})\cap A^m_{{\alpha,\beta}}({I})$ and $ \widehat{\partial}_{t}v \in A^{m-1}_{{\alpha,\beta}}({I})$, we have
\begin{equation}\label{I_thm_1}
\|\mathcal{I}^{\alpha,\beta}_N v-v\|_{\chi^{\alpha,\beta}}\leq c\sqrt{\frac{(N+1-\widetilde{m})!}{(\beta+1)^{\widetilde{m}-\alpha}N!}}\left\{c^{\beta}_1\|\widehat{\partial}_{t}^{\,\widetilde{m}}{v}\|_{\chi^{\alpha+m-1,\beta}}+c^{\beta}_2\sqrt{\log N} \|\widehat{\partial}_{t}^{\,\widetilde{m}}{v}\|_{\chi^{\alpha+m,\beta}}\right\},
\end{equation}
where $c^{\beta}_1={(\beta+1)^{-\frac{1}{2}}},\quad c^{\beta}_2=2\sqrt{\max\{1,\beta+1\}}$ and   $\widetilde{m}=\min\{m,N+1\}$.
\end{thm}
\begin{proof} By the triangle inequality, we have
\begin{equation}\begin{aligned}\label{a0}
\|\mathcal{I}^{\alpha,\beta}_N v-v\|_{\chi^{\alpha,\beta}}\leq \|\mathcal{I}^{\alpha,\beta}_N v-\pi^{\alpha,\beta}_Nv\|_{\chi^{\alpha,\beta}}+\|\pi^{\alpha,\beta}_Nv-v\|_{\chi^{\alpha,\beta}}.
\end{aligned}\end{equation}
We only need to estimate the first term since the estimate for the  second term is already available in Theorem \ref{thm_Proj}.
Thanks to  \eqref{stability},
\begin{equation}\begin{aligned}\label{a1}
 \|\mathcal{I}^{\alpha,\beta}_N v-&\pi^{\alpha,\beta}_Nv\|_{\chi^{\alpha,\beta}}= \|\mathcal{I}^{\alpha,\beta}_N (v-\pi^{\alpha,\beta}_Nv)\|_{\chi^{\alpha,\beta}}
 \\&\leq c\left(c^{\beta}_1N^{-\frac{1}{2}}\|\widehat{\partial}_t (v-\pi^{\alpha,\beta}_Nv)\|_{\chi^{\alpha,\beta}}+c^{\beta}_2\sqrt{\log{N}}\|v-\pi^{\alpha,\beta}_Nv\|_{A^1_{\alpha,\beta}}\right).
\end{aligned}\end{equation}
The term $\|v-\pi^{\alpha,\beta}_Nv\|_{A^\mu_{\alpha,\beta}}~(\mu=0,1)$ can be estimated through Theorem \ref{thm_Proj}. For the first term in the last inequality, we have
\begin{equation}\label{a2}
\|\widehat{\partial}_t (v-\pi^{\alpha,\beta}_Nv)\|_{\chi^{\alpha,\beta}}\leq \|\widehat{\partial}_t v-\pi^{\alpha,\beta}_N\{\widehat{\partial}_tv\}\|_{\chi^{\alpha,\beta}}+\|\pi^{\alpha,\beta}_N\{\widehat{\partial}_tv\}-\widehat{\partial}_t \{\pi^{\alpha,\beta}_Nv\})\|_{\chi^{\alpha,\beta}}.
\end{equation}
We now follow a classical procedure as in \cite{bernardi1997spectral,ben1998spectral} to derive the desired estimate. The starting point is the relation
\begin{equation*}
\widehat{\partial}_t v=\sum_{n=0}^\infty \hat{v}^{\alpha,\beta}_n \widehat{\partial}_t \mathcal{S}^{(\alpha,\beta)}_n\overset{\eqref{S_Derivative}}{=}\sum_{n=1}^\infty \hat{v}^{\alpha,\beta}_n \left((\beta+1)\sum_{l=0}^{n-1} \mathcal{S}^{(\alpha,\beta)}_l\right){=}\sum_{l=0}^\infty \left((\beta+1)\sum_{n=l+1}^{\infty}\hat{v}^{\alpha,\beta}_n \right) \mathcal{S}^{(\alpha,\beta)}_l.
\end{equation*}
The above equation implies that
\begin{equation*}
\pi^{\alpha,\beta}_N\{\widehat{\partial}_t v\}(t)=\sum_{n=0}^N \hat{v}^{\alpha,\beta}_{1,n} \mathcal{S}^{(\alpha,\beta)}_n(t),\quad \hat{v}^{\alpha,\beta}_{1,n}:=(\beta+1)\sum_{k=n+1}^{\infty}\hat{v}^{\alpha,\beta}_k.
\end{equation*}
Similarly, we have
\begin{equation*}
\widehat{\partial}_t\{\pi^{\alpha,\beta}_N v\}(t)=\sum_{n=0}^{N-1} \left((\beta+1)\sum_{k=n+1}^{N}\hat{v}^{\alpha,\beta}_k \right) \mathcal{S}^{(\alpha,\beta)}_n(t)=\sum_{n=0}^{N-1}(\hat{v}^{\alpha,\beta}_{1,n}-\hat{v}^{\alpha,\beta}_{1,N})\mathcal{S}^{(\alpha,\beta)}_n(t).
\end{equation*}
Hence,
\begin{equation}\begin{aligned}
\|\pi^{\alpha,\beta}_N\{\widehat{\partial}_tv\}-\widehat{\partial}_t \{\pi^{\alpha,\beta}_Nv\})\|^2_{\chi^{\alpha,\beta}}&\overset{\eqref{S_orth}}{=}\sum_{n=0}^N \gamma^{(\alpha,\beta)}_n(\hat{v}^{\alpha,\beta}_{1,N})^2=\gamma^{(\alpha,\beta)}_N(\hat{v}^{\alpha,\beta}_{1,N})^2\sum_{n=0}^N \gamma^{(\alpha,\beta)}_n(\gamma^{(\alpha,\beta)}_N)^{-1}
\\&\leq \|\widehat{\partial}_t v-\pi^{\alpha,\beta}_{N-1}\{\widehat{\partial}_tv\}\|_{\chi^{\alpha,\beta}}\sum_{n=0}^N \gamma^{(\alpha,\beta)}_n(\gamma^{(\alpha,\beta)}_N)^{-1}\label{a3}.
\end{aligned}\end{equation}
It remains  to estimate $s_N:=\sum_{n=0}^N \gamma^{(\alpha,\beta)}_n(\gamma^{(\alpha,\beta)}_N)^{-1}$. For any $\alpha\geq 0$, in view of the expression of $\gamma^{(\alpha,\beta)}_n$, it's obvious that $s_N\leq N+1$. For $-1<\alpha<0$, we use Stirling formula to deduce that for a large integer $M\leq k\leq N$,
$$\frac{\gamma^{(\alpha,\beta)}_k}{\gamma^{(\alpha,\beta)}_N}=\dfrac{\Gamma(N+1)\Gamma(k+\alpha+1)}{\Gamma(N+\alpha+1)\Gamma(k+1)}\sim N^{-\alpha}k^{\alpha}.$$
Therefore, there exists a constant $c_M$ such that
\begin{equation}\label{a4}
s_N=\sum_{n=0}^N \gamma^{(\alpha,\beta)}_n(\gamma^{(\alpha,\beta)}_N)^{-1}\leq N^{-\alpha}(c_M+c\sum_{k=M}^N k^{\alpha})\leq cN.
\end{equation}
Finally, combing \eqref{a0}-\eqref{a4}  and Theorem \ref{thm_Proj} leads to \eqref{I_thm_1}.
\end{proof}

\begin{rem}\label{Coro_Gauss_estimate} Let $\{{t}_j^{(\alpha,\beta)}\}_{j=0}^N$ and $\{{\omega}_j^{(\alpha,\beta)}\}_{j=0}^N$ be the same as \eqref{node_weight}. Then, we have the following estimate for the quadrature error:
\begin{equation}\label{I_rem}
 \big|\int_{0}^1 v(t) \chi^{\alpha,\beta}(t) {\rm d}t-\sum_{j=0}^N v({t}_j^{(\alpha,\beta)})\,{\omega}_j^{(\alpha,\beta)}\big|\leq \frac{\Gamma(\alpha+1)}{(\beta+1)^{\alpha+1}}~\|\mathcal{I}^{\alpha,\beta}_N v-v\|_{\chi^{\alpha,\beta}}.
\end{equation}
Indeed, the above estimate can be derived from
$$\sum_{j=0}^N v({t}_j^{(\alpha,\beta)})\,{\omega}_j^{(\alpha,\beta)}=\int_0^1 \mathcal{I}^{\alpha,\beta}_N v(t) \chi^{\alpha,\beta}(t){\rm d}t,$$
and
$$\int_0^1(-\log t)^\alpha  t^{\beta}{\rm d} t=\int_0^\infty y^\alpha e^{-{(\beta+1)}y}{\rm d} y=\frac{\Gamma(\alpha+1)}{(\beta+1)^{\alpha+1}}.$$ \qed
\end{rem}

\begin{figure}[htp!]
\begin{minipage}{0.495\linewidth}
\begin{center}
\includegraphics[scale=0.4]{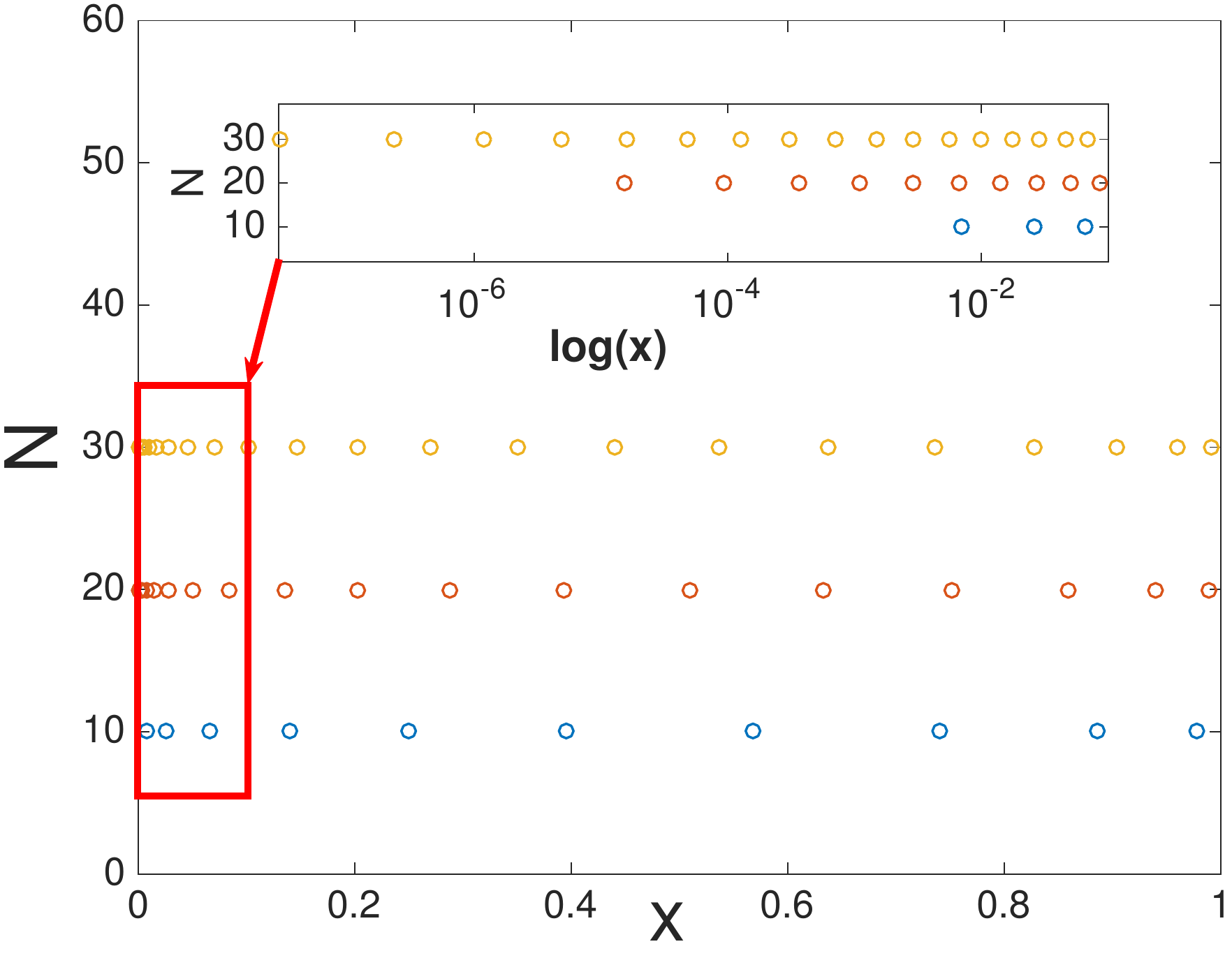}
\end{center}
\end{minipage}
\begin{minipage}{0.495\linewidth}
\begin{center}
\includegraphics[scale=0.4]{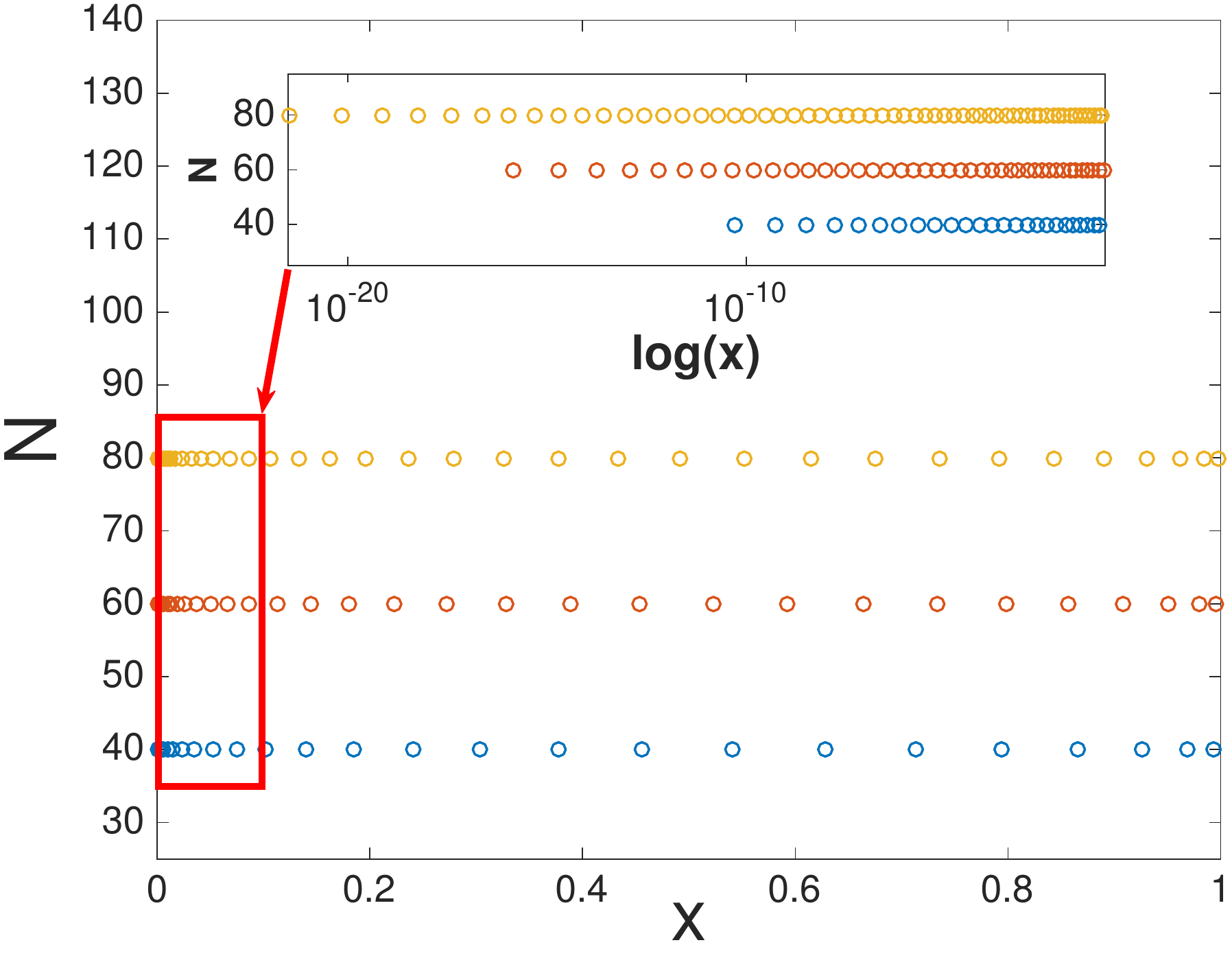}
\end{center}
\end{minipage}
\caption{Nodes distribution of  $\mathcal{S}^{(\alpha,\beta)}_n(t)$: $\alpha=0,~\beta=5$ with different $N$. }\label{graph_nodes1}
\end{figure}

 \begin{figure}[htp!]
\begin{minipage}{0.495\linewidth}
\begin{center}
\includegraphics[scale=0.4]{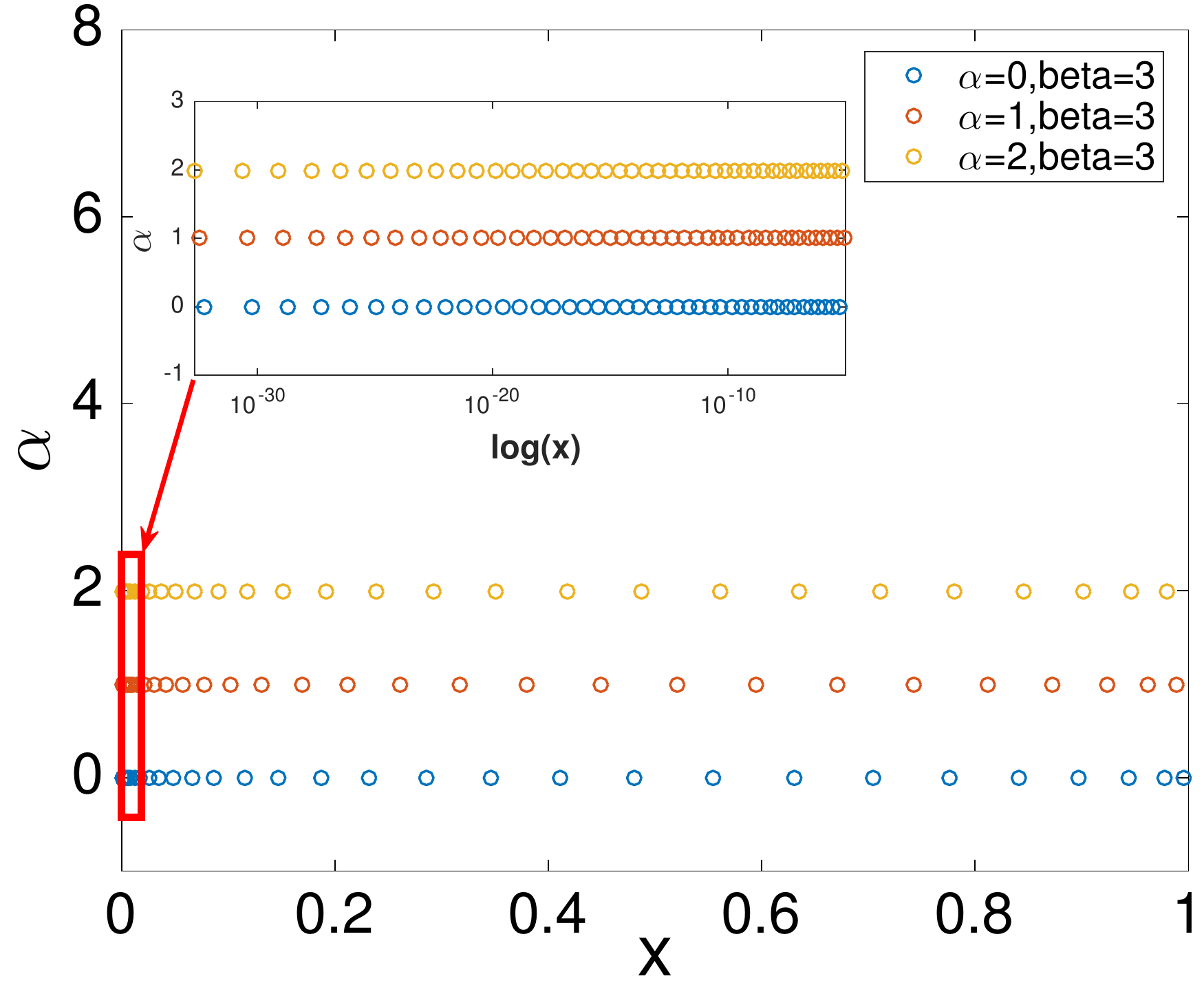}
\end{center}
\end{minipage}
\begin{minipage}{0.495\linewidth}
\begin{center}
\includegraphics[scale=0.4]{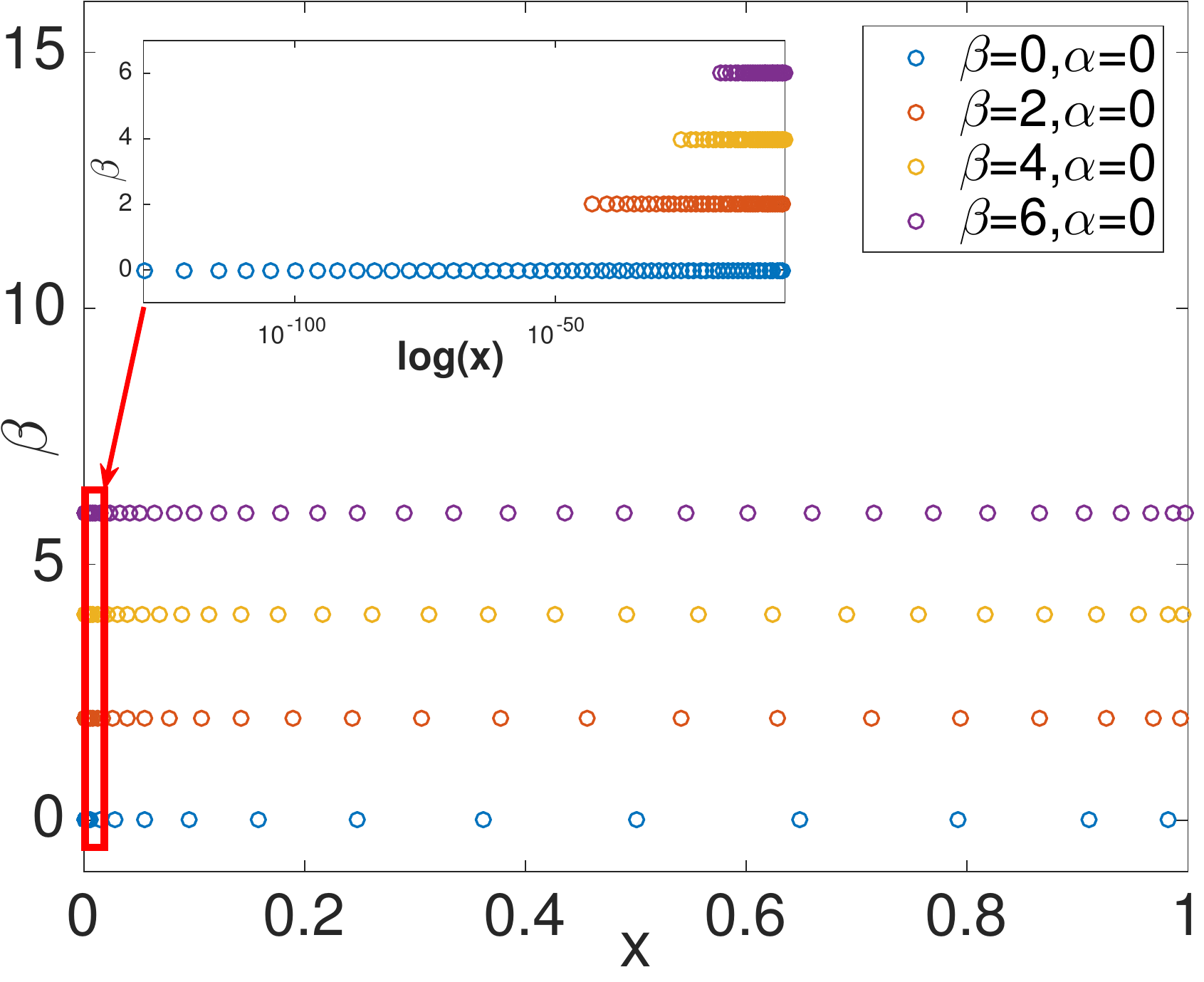}
\end{center}
\end{minipage}
\caption{  Nodes distribution of  $\mathcal{S}^{(\alpha,\beta)}_n(t)$: $N=80$ with different $\alpha,\,\beta$. }\label{graph_nodes2}
\end{figure}

To understand better why the singular function $t^s,~s>0$ can be well approximated by LOFs, we plot distribution of the Gauss-LOFs quadrature nodes  $\{t_j^{(\alpha,\beta)}\}_{j=0}^N$ with various $N$ and $\alpha,~\beta$ in Fig. \ref{graph_nodes1} and  \ref{graph_nodes2}. We observe from Fig. \ref{graph_nodes1} that
the nodes cluster near zero,  with significant more points near zero than the usual Gauss-Radau points. Fig. \ref{graph_nodes2}  exhibits the influence of the parameters $(\alpha,\beta)$ on the nodes distribution.
In particular, as $\alpha$ increases with $\beta$ fixed, the quadrature nodes move towards zero; on the other hand, as $\beta$ increases with $\alpha$ fixed, the quadrature nodes move away from zero.

\subsection{Numerical examples}

\begin{figure}[htp!]
\begin{minipage}{0.495\linewidth}
\begin{center}
\includegraphics[scale=0.4]{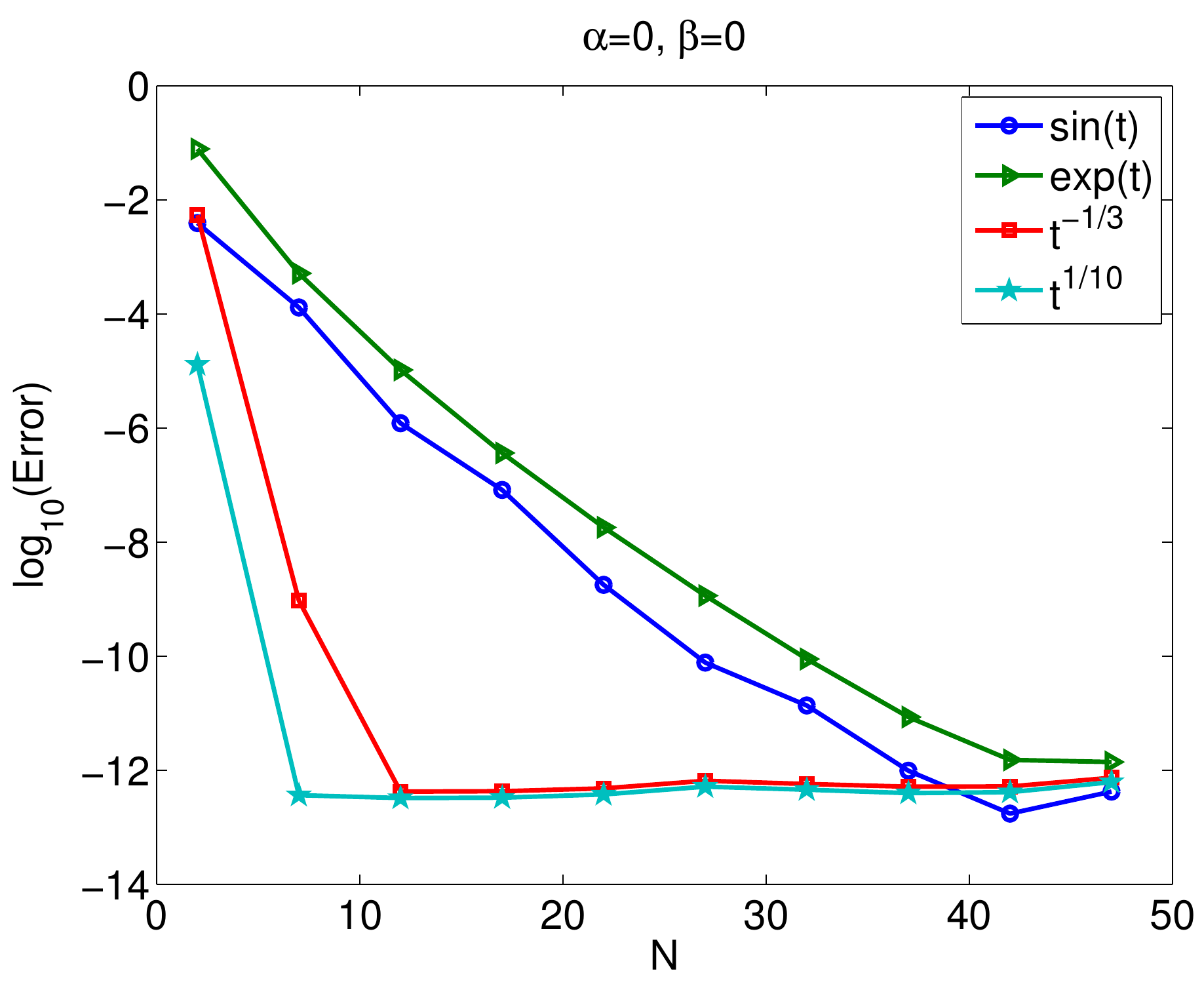}
\end{center}
\end{minipage}
\begin{minipage}{0.495\linewidth}
\begin{center}
\includegraphics[scale=0.4]{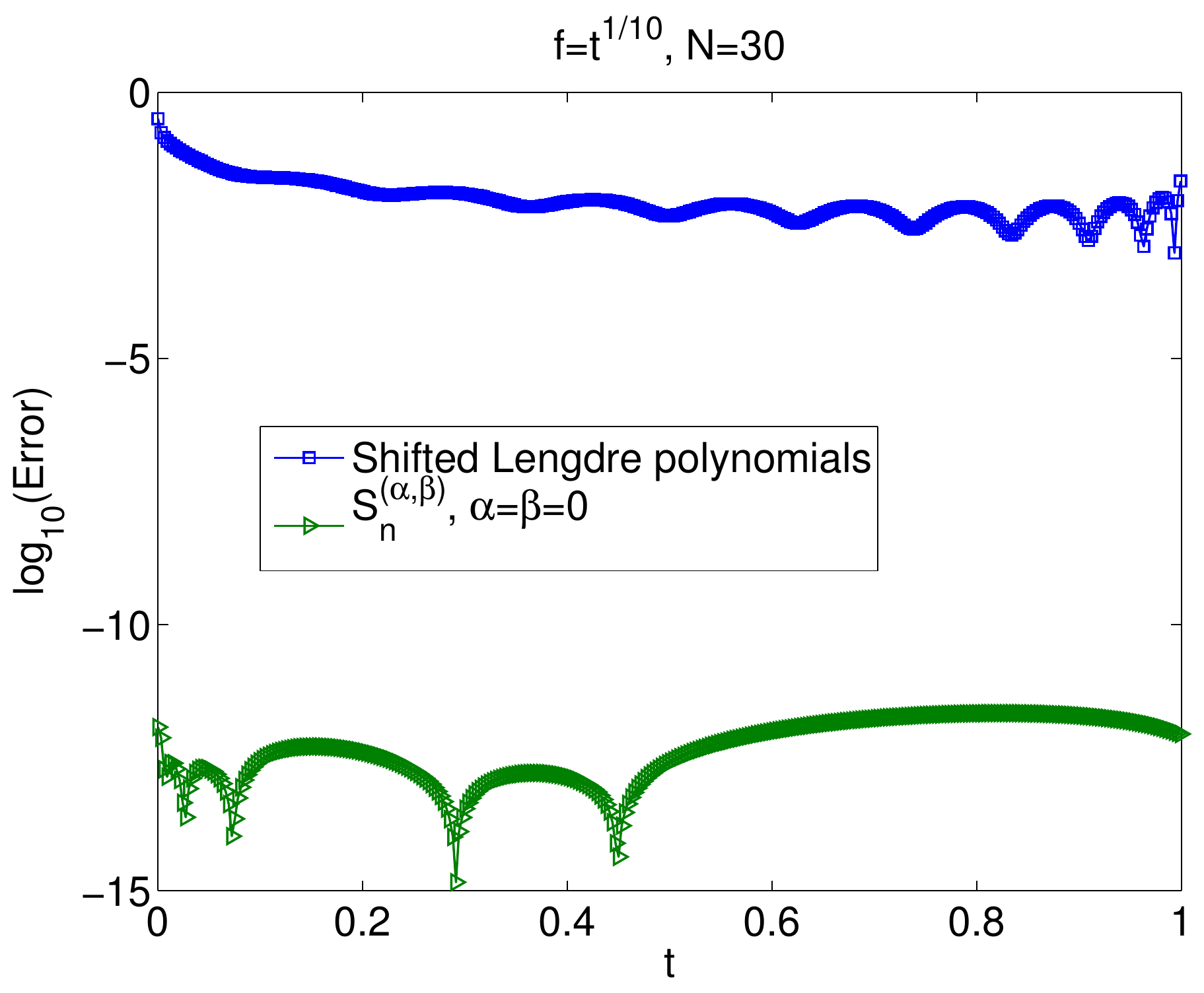}
\end{center}
\end{minipage}
\caption{ Left: Gauss-LOFs  quadrature errors.  Right: projection errors.}\label{Log_quad_expan}
\end{figure}
We first demonstrate the accuracy of Gauss-LOFs quadrature for computing
\begin{equation}\label{integral_GQ}
\int_0^1 f(t)(-\log{t})^\alpha t^\beta {\rm d}t,\quad \alpha, \beta>-1,
\end{equation}
with  the following functions
 $f(t)=\sin{t},~e^{t},~t^{-1/3}$ and $t^{1/10}$, respectively.
The quadrature errors are shown in the left of Fig. \ref{Log_quad_expan}.
We observe that the errors decay exponentially in all cases. We note that
 $f(t)=t^{-1/3}$ is singular and  can not be computed efficiently by the classical Gauss quadrature. However, $t^{-1/3}$ is smooth with the norm defined  through the pseudo-derivative \eqref{pseudo_deri},  so we  achieve exponential convergence for this case as well.

Next, we compute the projection error for  $f(t)=t^{1/10}$ which is not smooth in the usual Sobolev space, but it is smooth with the norm defined  through the pseudo-derivative.
 In the right of Fig. \ref{Log_quad_expan}, we plot the projection errors by using  the shifted Legendre polynomial $L_n(2t-1), ~t\in {I}$ and LOFs for function $f(t)=t^{1/10}$ with the fixed degree of basis $N=30$. We observe that the projection error by using LOFs is uniformly small across the  interval [0,1], while the error by using the shifted Legendre polynomial
 is very large.

\section{Generalized Log orthogonal functions}
The LOFs introduced in the last section is capable of resolving certain type
of singularities at $t=0$, but
 LOFs $\mathcal{S}^{(\alpha,\beta)}_n(t)$ consist of $\{(-\log{t})^k\}_{k=0}^{n}$, so  grow very fast near $t=0$ (cf. Fig. \ref{Value_orth1}). This behavior may severely  affects the accuracy in many situations.
 In addition,   derivatives of LOFs involve  the singular term $t^{-1}$, so they are not suitable to serve  as  basis functions  to represent solutions of PDEs or fractional PDEs.

\begin{figure}[htp!]
\begin{minipage}{0.495\linewidth}
\begin{center}
\includegraphics[scale=0.4]{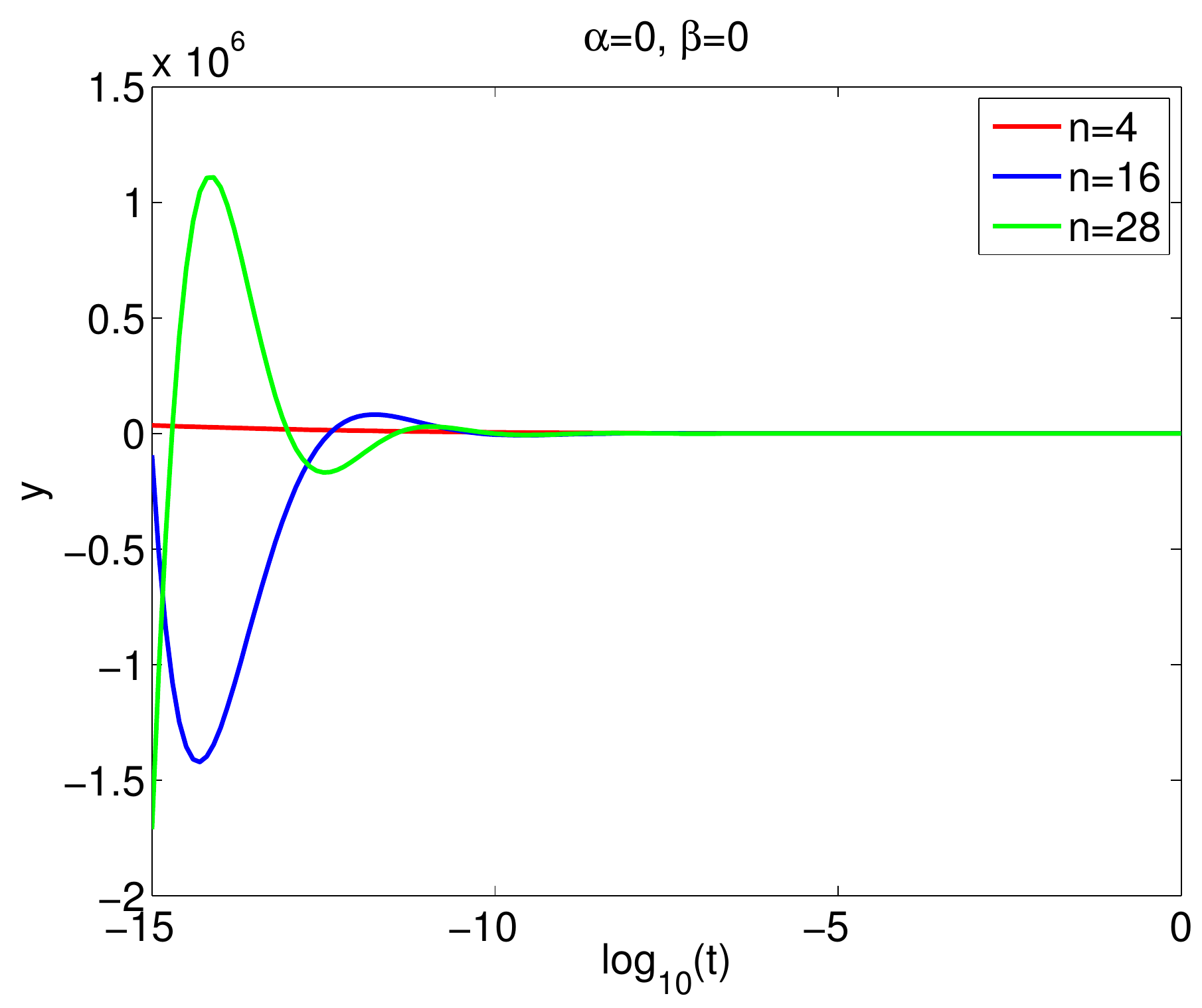}
\end{center}
\end{minipage}
\begin{minipage}{0.495\linewidth}
\begin{center}
\includegraphics[scale=0.4]{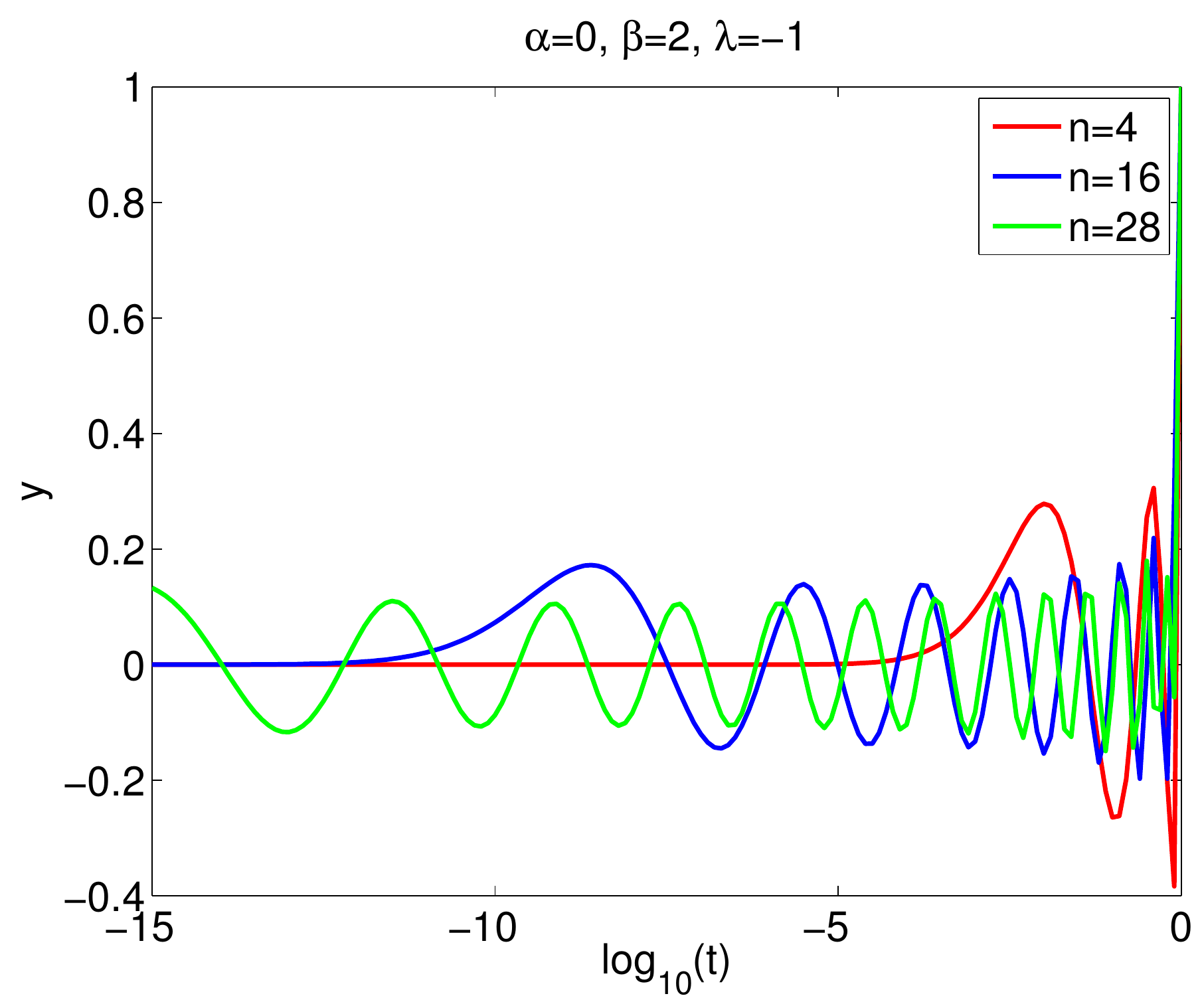}
\end{center}
\end{minipage}
\caption{ Left: Graphs of  $y=\mathcal{S}^{(\alpha,\beta)}_n(t)$. \qquad Right: Graphs of $y={\mathcal{S}}^{(\alpha,\beta,\lambda)}_n(t)$.}\label{Value_orth1}
\end{figure}
Therefore, we shall introduce below the generalized Log orthogonal functions (GLOFs) which are more suitable for numerical approximations of functions with weak singularities at one endpoint.

\subsection{Definition and properties}

\begin{rem}
As depicted in the left of Fig. \ref{Value_orth1}, values of
 LOFs near $t=0$ vary in a very large range. However, as shown in the right of Fig. \ref{Value_orth1}, GLOFs are much better behaved.
\qed
\end{rem}

\begin{defn}[GLOFs]\label{def_G_orth}
{\em
Let $\alpha,\beta>-1$, $\lambda\in\mathbb{R}$. We define the GLOFs by
\begin{equation}\label{G_orth}
{\mathcal{S}}^{(\alpha,\beta,\lambda)}_n(t):=t^{\frac{\beta-\lambda}{2}} {\mathcal{S}}^{(\alpha,\beta)}_n(t),\quad n\geq 0.
\end{equation}
In particular, ${\mathcal{S}}^{(\alpha,\beta,\beta)}_n(t)={\mathcal{S}}^{(\alpha,\beta)}_n(t)$.
}\end{defn}


GLOFs enjoy similar properties as those listed in Lemma \ref{lem:prop} for LOFs. Owing to the relation \eqref{S_orth}, it is obvious that
\begin{equation}\label{GS_orth}
    \int_{0}^1 \mathcal{S}_n^{(\alpha,\beta,\lambda)}(t)~  \mathcal{S}_m^{(\alpha,\beta,\lambda)}(t)\,~ (-\log{t})^{\alpha} \,t^\lambda \, {\rm d}t=\gamma^{(\alpha,\beta)}_n \delta_{mn},\end{equation}
    where $ \gamma^{(\alpha,\beta)}_n$ is the same as the definition in \eqref{S_orth}.

The derivative relation can be derived from the relation \eqref{S_Derivative} and  Definition \ref{def_G_orth}. Indeed,
\begin{equation}\begin{aligned}\label{D_G}
\partial_t {\mathcal{S}}^{(\alpha,\beta,\lambda)}_n(t)&=\frac{\beta-\lambda}{2}t^{\frac{\beta-\lambda}{2}-1}{\mathcal{S}}^{(\alpha,\beta)}_n(t)+t^{\frac{\beta-\lambda}{2}}\partial_t{\mathcal{S}}^{(\alpha,\beta)}_n(t)
\\&=\frac{\beta-\lambda}{2}t^{\frac{\beta-\lambda-2}{2}}{\mathcal{S}}^{(\alpha,\beta)}_n(t)+(\beta+1)t^{\frac{\beta-\lambda-2}{2}}{\mathcal{S}}^{(\alpha+1,\beta)}_{n-1}(t)
\\&=\frac{\beta-\lambda}{2}{\mathcal{S}}^{(\alpha,\beta,\lambda+2)}_n(t)+(\beta+1){\mathcal{S}}^{(\alpha+1,\beta,\lambda+2)}_{n-1}(t).
\end{aligned}\end{equation}

The  pseudo-derivative with respect to GLOFs should be defined as
\begin{equation}\label{pseudo2}
\widehat{\partial}_{\gamma,t}u=t^{1+\gamma}\partial_t \{t^{-\gamma} u\}.
\end{equation}
Then, thanks to the definition of $\mathcal{S}^{(\alpha,\beta,\lambda)}_n$ and  \eqref{S_Derivative}, we have the following important derivative relation:
\begin{equation}\label{GS_Derivative}
(\beta+1)^{-1}\widehat{\partial}_{\frac{\beta-\lambda}{2},t}~\mathcal{S}_n^{(\alpha,\beta,\lambda)}(t)= \mathcal{S}_{n-1}^{(\alpha+1,\beta,\lambda)}(t)=\sum_{l=0}^{n-1}\mathcal{S}_{l}^{(\alpha,\beta,\lambda)}(t), \quad n\geq 1.
\end{equation}

Let
\begin{equation}\label{lag:lambda}
{t}_j^{(\alpha,\beta,\lambda)}:={t}_j^{(\alpha,\beta)},~{\chi}_j^{(\alpha,\beta,\lambda)}:=({t}_j^{(\alpha,\beta)})^{\lambda-\beta}{\chi}_j^{(\alpha,\beta)},\;{j=0,1,\cdots,N,}
\end{equation}
and denote
 \begin{equation}\label{G_PN}
 {\mathcal{P}}^{\gamma,\log{t}}_{N}:=\{t^{\gamma} p(t):~p\in{\mathcal{P}}^{\log{t}}_{N}\}.
 \end{equation}
Then, we have the following {\em Gauss-GLOFs quadrature}:
   \begin{equation}\label{Gauss_Quad}
  \int_{0}^1 f(t) (-\log{t})^\alpha t^{\lambda} {\rm d}t= \sum_{j=0}^N f({t}_j^{(\alpha,\beta,\lambda)})\,{\chi}_j^{(\alpha,\beta,\lambda)},\quad \forall \, f \in \mathcal{P}^{\beta-\lambda,\log{t}}_{2N+1}.
  \end{equation}

 In addition, we derive from Definition \ref{def_LOF} and the closed form of the Laguerre polynomial that
 \begin{equation*}\label{GLOFs}
{\mathcal{S}}^{(\alpha,\beta,\lambda)}_n(t)=\sum_{k=0}^n\frac{(-1)^k}{k!}{{n}\choose{n-k}}t^{\frac{\beta-\lambda}{2}}[-(\beta+1)\log{t}]^k,\quad t\in {I}:=(0,1).
\end{equation*}


\subsection{Projection estimate}
Let $\alpha,\beta>-1$ and $\lambda\in\mathbb{R}$.
 We define the projection operator
 $\pi_N^{\alpha,\beta,\lambda}$:   $L^2_{\chi^{\alpha,\lambda}} \rightarrow   \mathcal{P}^{\frac{\beta-\lambda}{2},\log{t}}_{N}$
by
\begin{equation}\label{projection2}
(u-\pi_N^{\alpha,\beta,\lambda} u, v)_{\chi^{\alpha,\lambda}}=0, \quad \forall u\in L^2_{\chi^{\alpha,\lambda}},\; v\in \mathcal{P}^{\frac{\beta-\lambda}{2},\log{t}}_{N},
\end{equation}
where $ \chi^{\alpha,\lambda}(t):=(-\log{t})^\alpha t^\lambda$.

Thanks to the orthogonality of the basis $\{\mathcal{S}^{(\alpha,\beta,\lambda)}_n\}_{n=0}^\infty$, we have
\begin{equation}\label{projection_expan}
\pi_N^{\alpha,\beta,\lambda} u=\sum_{n=0}^N \hat{u}^{\alpha,\beta,\lambda}_n \,\mathcal{S}^{(\alpha,\beta,\lambda)}_n \;\text{ with }\; \hat{u}^{\alpha,\beta,\lambda}_n=(\gamma^{(\alpha,\beta)}_n)^{-1}\int_{0}^1 u(t)\, \mathcal{S}^{(\alpha,\beta,\lambda)}_n(t) \,\chi^{\alpha,\lambda} (t) {\rm d} t.
\end{equation}

To better describe the approximability of  $\pi_N^{\alpha,\beta,\lambda}$, we  define  non-uniformly weighted Sobolev  spaces
\begin{equation}\label{HilbertSpace2}
{A}^{k}_{\alpha,\beta,\lambda}({I}):=\{v\in L^{2}_{\chi^{\alpha,\lambda}}({I}):~{\widehat{\partial}_{\frac{\beta-\lambda}{2},t}}^{\,j} v\in  L^{2}_{\chi^{\alpha+j,\lambda}}({I}),~j=1,2,\ldots,k\},\quad k\in \mathbb{N},
\end{equation}
with the corresponding semi-norm and norm defined by
\begin{equation*}
|v|_{{A}^{m}_{\alpha,\beta,\lambda}}:=\|\widehat{\partial}_{\frac{\beta-\lambda}{2},t}^{\,m}{v}\|_{\chi^{\alpha+m,\lambda}},\quad \|v\|_{{A}^{m}_{\alpha,\beta,\lambda}}:=\left(\sum_{k=0}^m |v|^2_{{A}^{k}_{\alpha,\beta,\lambda}}\right)^{\frac{1}{2}}.
\end{equation*}


\begin{thm}\label{thm_proj2}
Let $m,\, N,\,k\in \mathbb{N}$, $\lambda\in\mathbb{R}$ and $\alpha,\beta>-1$. For any $u\in{A}^{m}_{\alpha,\beta,\lambda}({I})$ and $0\leq k\leq\widetilde{m}=\min\{m,N+1\}$, we have
\begin{equation}\label{proj2}
\|\widehat{\partial}_{\frac{\beta-\lambda}{2}, t}^{\,k} (u-\pi_N^{\alpha,\beta,\lambda} u)\|_{\chi^{\alpha+k,\lambda}}\leq  \sqrt{(\beta+1)^{k-\widetilde{m}}\frac{(N-\widetilde{m}+1)!}{(N-{k}+1)!}}~ \|\widehat{\partial}_{\frac{\beta-\lambda}{2},t}^{\,\widetilde{m}}
 u\|_{\chi^{\alpha+\widetilde{m},\lambda}},
\end{equation}
where  $\widehat{\partial}_{\frac{\beta-\lambda}{2},t}$ is the pseudo-derivative defined in \eqref{pseudo2}.
\end{thm}
\begin{proof}
 For any $u\in{A}^{m}_{\alpha,\beta,\lambda}({I})$,  we can expand it as  $u=\sum_{n=0}^\infty \hat{u}^{\alpha,\beta,\lambda}_n \mathcal{S}_n^{(\alpha,\beta,\lambda)}$. Due to
  \begin{equation*}
 \widehat{\partial}_{\frac{\beta-\lambda}{2},t}^{\,l} ~\mathcal{S}^{(\alpha,\beta,\lambda)}_n(t)\overset{\eqref{GS_Derivative} }{=}(\beta+1)^l\mathcal{S}^{(\alpha+l,\beta,\lambda)}_{n-l}(t), \quad l\leq n,
 \end{equation*}
and  \eqref{GS_orth}, we have
$$\|\widehat{\partial}_{\frac{\beta-\lambda}{2},t}^{\,l} u\|^2_{\chi^{\alpha+l,\lambda}}=\sum_{n=l}^\infty (\beta+1)^{2l}\gamma^{(\alpha+l,\beta)}_{n-l} |\hat{u}^{\alpha,\beta,\lambda}_n|^2,\quad l\geq 1.$$
Then, by following the same procedure as in the proof of Theorem \ref{thm_Proj}, we can obtain the desired result \eqref{proj2}.
\end{proof}

\subsection{Interpolation estimate}
Let $\{t_j^{\alpha,\beta}\}_{j=0}^N$ be the same set of collocation points as
for the LOFs. We define the interpolation operator
 ${\mathcal{I}}^{\alpha,\beta,\lambda}_N: C(I)\rightarrow   \mathcal{P}^{\frac{\beta-\lambda}2,\log{t}}_{N}$  by
 $$({\mathcal{I}}^{\alpha,\beta,\lambda}_N v)(t_j^{\alpha,\beta})=v(t_j^{\alpha,\beta}), \;j=0,1,\cdots,N.$$
 It is easy to see that
\begin{equation}\label{interpolation_21}
{\mathcal{I}}^{\alpha,\beta,\lambda}_N v(t)=\sum_{j=0}^N {v(t^{(\alpha,\beta)}_j)} {l}^{\beta,\lambda}_j\big({y}(t)\big), \quad y(t)=-(\beta+1)\log t
\end{equation}
where $\{{l}^{\beta,\lambda}_j\}$ are the Lagrange "polynomials" defined by
\begin{equation}\label{interpolation_22}
{l}^{\beta,\lambda}_j\big(y(t)\big)=\frac{t^{\frac{\beta-\lambda}{2}}\prod\limits_{i\neq j} \log (t^{(\alpha,\beta)}_i/t) }{(t^{(\alpha,\beta)}_j)^{\frac{\beta-\lambda}{2}}\prod\limits_{i\neq j}\log (t^{(\alpha,\beta)}_i/t^{(\alpha,\beta)}_j)}.
\end{equation}
In view of \eqref{interpolation_1} and \eqref{interpolation_2}, we have
$${\mathcal{I}}^{\alpha,\beta,\lambda}_N v(t)=t^{\frac{\beta-\lambda}{2}}{\mathcal{I}}^{\alpha,\beta}_N\{t^{\frac{\lambda-\beta}{2}}  v(t)\}\in \mathcal{P}^{\frac{\beta-\lambda}2,\log{t}}_{N}.$$
Hence, we can derive the following result  from Theorem \ref{thm_Interpolation}.
\begin{thm}\label{thm_Interpolation2}
Let $m$ and $N$ be positive integers, $\alpha,\beta>-1$ and $\lambda\in\mathbb{R}$.  For any $v\in C({I})\cap  A^m_{{\alpha,\beta,\lambda}}({I})$ and $\widehat{\partial}_{\frac{\beta-\lambda}{2},t}{v}\in A^{m-1}_{{\alpha,\beta,\lambda}}({I})$, we have
\begin{equation*}\label{I_thm_21}
\|\mathcal{I}^{\alpha,\beta,\lambda}_N v-v\|_{\chi^{\alpha,\lambda}}\leq c\sqrt{\frac{(N+1-\widetilde{m})!}{(\beta+1)^{\widetilde{m}-\alpha}N!}}\left\{c_{1}^{\beta}\|\widehat{\partial}_{\frac{\beta-\lambda}{2},t}^{\,\widetilde{m}}{v}\|_{\chi^{\alpha+m-1,\lambda}}+c_{2}^{\beta}\sqrt{\log N} \|\widehat{\partial}_{\frac{\beta-\lambda}{2},t}^{\,\widetilde{m}}{v}\|_{\chi^{\alpha+m,\lambda}}\right\}
\end{equation*}
where $c_{1}^{\beta}={(\beta+1)^{-\frac{1}{2}}},\quad c_{2}^{\beta}=2\sqrt{\max\{1,\beta+1\}}$ and $\widetilde{m}=\min\{m,N+1\}$.
\end{thm}
\begin{proof}
Since
$$\|\mathcal{I}^{\alpha,\beta,\lambda}_N v-v\|_{\chi^{\alpha,\lambda}}=\|\mathcal{I}^{\alpha,\beta}_N \{t^{\frac{\lambda-\beta}{2}}v\}-t^{\frac{\lambda-\beta}{2}}v\|_{\chi^{\alpha,\beta}},$$
and
$$\widehat{\partial}_t \{t^{\frac{\lambda-\beta}{2}}v\}=t^{\frac{\lambda-\beta}{2}}\widehat{\partial}_{\frac{\beta-\lambda}{2},t} v~\Longrightarrow~\widehat{\partial}^{\widetilde{m}}_t \{t^{\frac{\lambda-\beta}{2}}v\}=t^{\frac{\lambda-\beta}{2}}\widehat{\partial}^{\widetilde{m}}_{\frac{\beta-\lambda}{2},t} v.$$
We can then derive the desired result  from the above relation and Theorem \ref{thm_Interpolation}.
\end{proof}

{
\subsection{Explicit error estimate for a class of  weakly singular functions}\label{sec3.3}
The result in Theorem \ref{thm_Interpolation2} is not easy to interpret for general functions, so we 
consider the following typical   weakly singular functions
$$f(t)= t^r(-\log{t})^k,\quad r\geq 0,~k\in\mathbb{N}_0.$$
We first present a very useful relation of the Laguerre polynomials $\mathscr{L}^{(\alpha)}_n(y),\,\alpha>-1$.
\begin{lmm}
Let $s>0$, $\alpha>-1$. For $k,n\in \mathbb{N}_0$ and $n> k$, there exists 
\begin{equation}\label{LemmaSF}
\int_{0}^\infty y^{\alpha+k} e^{-sy} \mathscr{L}^{(\alpha)}_n(y) {\rm d} y=(\frac{s-1}{s})^n\frac{k!}{s^{\alpha+k+1}}\sum_{j=0}^k\frac{ \Gamma(n-j+k+\alpha+1)}{(j!)^2~\Gamma(n-j+1)}(\frac{s}{1-s})^j.
\end{equation}
\end{lmm}
\begin{proof}
The case $k=0$ is a direct result of \cite[7.414: 8]{Gradshteyn2007In}, i.e., 
$$\int_{0}^\infty e^{-sy}y^\alpha \mathscr{L}^{(\alpha)}_n(y) {\rm d} y=\frac{\Gamma(\alpha+n+1)}{\Gamma(n+1)}(1-1/s)^n s^{-\alpha-1}.$$
For $k>0$, owing to Rodrigues' formula (see Szego \cite[(5.1.5)]{szego1975orthogonal} ), we have that
\begin{equation}\begin{aligned}\label{LemmaSF2}
\int_{0}^\infty y^{k}&e^{-sy}y^{\alpha}\mathscr{L}^{(\alpha)}_n(y) {\rm d} y=\frac{1}{n!}\int_{0}^\infty y^{k} e^{(1-s)y}~ \partial_y^n(y^{n+\alpha}e^{-y}) {\rm d} y\\
&=\frac{(-1)^n}{n!}\int_{0}^\infty  \partial_y^n(y^{k} e^{(1-s)y})~y^{n+\alpha}e^{-y} {\rm d} y\\
&=\frac{(-1)^n}{n!}\sum_{j=0}^k {{n}\choose{j}} (1-s)^{n-j} \int_{0}^\infty \frac{ k!}{j!} y^{k-j} ~y^{n+\alpha}e^{-sy} {\rm d} y\\
&=(-1)^n\sum_{j=0}^k\frac{ k! \Gamma(n-j+k+\alpha+1)}{(j!)^2\Gamma(n-j+1)}\frac{(1-s)^{n-j}}{s^{n-j+\alpha+k+1}}.\\
\end{aligned}\end{equation}
One can easily check the equivalence of  the relations \eqref{LemmaSF} and \eqref{LemmaSF2}, which completes the proof. 
\end{proof}
 
 With the above lemma in hand, we have the following error estimate:
 \begin{thm}\label{estimateSF}
 Given  
 $f(t)=t^r(-\log{t})^k,~ r\geq 0,~k\in\mathbb{N}_0.$
 Let $\lambda>-1-2r$,   $\alpha,\beta >-1$ and $\beta>\lambda$. Then, we have 
 $$ f\in L^2_{\chi^{\alpha,\lambda}}\;\text{ and }\; R_{\alpha,\beta,\lambda}=\left |\frac{2r+\lambda-\beta}{2r+2+\lambda+\beta}\right| < 1,$$
and 
 \begin{equation}\label{ProjectionSF}
\|f-\pi_N^{\alpha,\beta,\lambda} f\|_{\chi^{\alpha,\lambda}} \leq c ~(k+1)! ~N^{\frac{\alpha+1}{2}+k}~(R_{\alpha,\beta,\lambda})^N \;\text{ when }\; N>-\dfrac{2k+\alpha+2}{2\log(R_{r,\beta,\lambda})},
\end{equation}
where 
$$c\approx  \sqrt{\dfrac{2^{\alpha+1+k}(\beta+1)^{2\alpha+2-k}}{(\beta+\lambda+2r+2)^{\alpha+1+k}}}.$$
 \end{thm}
 \begin{proof}
Since $\lambda>-1-2r$, it is easy to check that $f\in L^2_{\chi^{\alpha,\lambda}}$ and 
$R_{\alpha,\beta,\lambda}=\left |\frac{2r+\lambda-\beta}{2r+2+\lambda+\beta}\right| < 1$.
\\
Thanks to the orthogonality of the basis $\{\mathcal{S}^{(\alpha,\beta,\lambda)}_n\}_{n=0}^\infty$, we can write\\
$$f=\sum_{n=0}^\infty \hat{f}^{\alpha,\beta,\lambda}_n \,\mathcal{S}^{(\alpha,\beta,\lambda)}_n,\quad \pi_N^{\alpha,\beta,\lambda} f=\sum_{n=0}^N \hat{f}^{\alpha,\beta,\lambda}_n \,\mathcal{S}^{(\alpha,\beta,\lambda)}_n$$
with coefficients
\begin{equation*}\label{projection_coeffcient}
\hat{f}^{\alpha,\beta,\lambda}_n=(\gamma^{(\alpha,\beta)}_n)^{-1}\int_{0}^1 f(t)\, \mathcal{S}^{(\alpha,\beta,\lambda)}_n(t) \,\chi^{\alpha,\lambda} (t) {\rm d} t,\quad n=0,1,\ldots,N.
\end{equation*}
Let $y(t)=-(\beta+1)\log{t}$ and ${\rm d}y=-(\beta+1)t^{-1}{\rm d} t$. It holds that
\begin{equation*}\begin{aligned}
\hat{f}^{\alpha,\beta,\lambda}_n&=\frac{(\beta+1)^{\alpha+1}\Gamma(n+1)}{\Gamma(n+\alpha+1)}\int_{0}^1t^r(-\log{t})^k\,t^{(\beta-\lambda)/2} \mathscr{L}^{(\alpha)}_n(y(t)) \,(-\log{t})^\alpha t^{\lambda} {\rm d} t\\
&=\frac{(\beta+1)^{-k}\Gamma(n+1)}{\Gamma(n+\alpha+1)}\int_{0}^\infty \exp(-\frac{\beta+\lambda+2r+2}{2(\beta+1)}y)y^{\alpha+k} \mathscr{L}^{(\alpha)}(y) {\rm d} y.
\end{aligned}\end{equation*}
Taking $s=({\beta+\lambda+2r+2})/{(2\beta+2)}$  into \eqref{LemmaSF}, we have
\begin{equation}\label{P0SF}
\hat{f}^{\alpha,\beta,\lambda}_n=(\frac{s-1}{s})^n~ \frac{\Gamma(n+1)~k!(\beta+1)^{-k} }{\Gamma(n+\alpha+1)~s^{\alpha+k+1}}\sum_{j=0}^k\frac{ \Gamma(n-j+k+\alpha+1)}{(j!)^2~\Gamma(n-j+1)}(\frac{s-1}{s})^{-j}.
\end{equation}
 Owing to \cite[Lemma 2.1]{ZhWX13}, we have
\begin{equation}\label{P1SF}
\frac{\Gamma(n+a)}{\Gamma(n+b)}\le \nu_n^{a,b} n^{a-b},\quad n+a>1~ \text{and} ~n+b>1,
\end{equation}
where
\begin{equation}\label{P2SF}
\nu_n^{a,b}=\exp\Big(\frac{a-b}{2(n+b-1)}+\frac{1}{12(n+a-1)}+\frac{(a-b)^2}{n}\Big).
\end{equation}
Combing \eqref{P0SF}-\eqref{P2SF} and the fact that $R_{r,\beta,\lambda}<1$, we have
\begin{equation*}\begin{aligned}
\|f-\pi_N^{\alpha,\beta,\lambda} f\|^2_{\chi^{\alpha,\lambda}} =&\sum_{N+1}^\infty |\hat{f}^{\alpha,\beta,\lambda}_n|^2 \gamma_n^{\alpha,\beta} \leq C_{\alpha,\beta}^{\lambda,r}\sum_{N+1}^\infty  \nu^{\alpha,0}_n n^{2k+\alpha}(R_{r,\beta,\lambda})^{2n-2k}
\\ \leq&C_{\alpha,\beta}^{\lambda,r}\nu^{\alpha,0}_N \int_N^\infty (R_{r,\beta,\lambda})^{2x-2k}x^{2k+\alpha} {\rm d}x.
\end{aligned}\end{equation*}
where $C_{\alpha,\beta}^{\lambda,r}=\dfrac{2^{\alpha+1+k}(\beta+1)^{2\alpha+2-k}}{(\beta+\lambda+2r+2)^{\alpha+1+k}}\big((k+1)!\big)^2$. 
Finally, as $a^{2x-2k}x^{2k+\alpha+2}$ is a decreasing function of $x$ when $N>-\dfrac{2k+\alpha+2}{2\log(R_{r,\beta,\lambda})}$, we conclude
that
  \begin{equation*}\begin{aligned}
\|f-\pi_N^{\alpha,\beta,\lambda} f\|_{\chi^{\alpha,\lambda}} \leq \sqrt{C_{\alpha,\beta}^{\lambda,r}}~  N^{\frac{\alpha+1}{2}+k}~(R_{r,\beta,\lambda})^{N-k} \;\text{ when }\; N>-\dfrac{2k+\alpha+2}{2\log(R_{r,\beta,\lambda})}.
\end{aligned}\end{equation*}
The proof is complete.
\end{proof}
 
The above theorem provides an accurate estimate for the GLOFs to a large class of  singular functions. In particular, by setting $\alpha=\lambda=0$, we have an estimate in $L^2$-norm.  

\begin{cor}\label{singularity_appro}
For $f(t)=t^r (-\log{t})^k,~r\geq 0,~k\in \mathbb{N}$,  it holds that
 \begin{equation}\label{EstimateSF_1}
 \|f-\pi_N^{0,\beta,0} f\| \leq {\sqrt{2}^k }(\beta+1)^{-k} k! N^k \sqrt{2(\beta+1)N}~\Big|\frac{2r-\beta}{2r+\beta+2}\Big|^{N-k}.
 \end{equation} 
 In particular, for $f=t^r,~r\geq 0$,  we have
 \begin{equation}\label{EstimateSF_2}
 \|f-\pi_N^{0,\beta,0} f\| \leq \sqrt{2(\beta+1)N}~\Big|\frac{2r-\beta}{2r+\beta+2}\Big|^{N}.
 \end{equation} 
 \qed
 \end{cor}
 
 In order to verify the above theoretical results, we plot the error curves for the GLOFs approximation to $f(t)=t^r$ with various $r$ in Fig. \ref{graph_GLOFsApproxi}, left with $r\in (0,1)$ and right with $r$ being integers. We observe  exponential convergence for all $r\geq 0$. We also plot the error curves for the GLOFs approximation to $f(t)=t (-\log{t})^k$ and $f(t)=t^2 (-\log{t})^k$ in Fig. \ref{graph_GLOFsApproxi2}. We also observe  exponential convergence in all cases. All these numerical results are  consistent with the approximation results in Corollary \ref{singularity_appro}.
 
 \begin{figure}[htp!]
\begin{minipage}{0.495\linewidth}
\begin{center}
\includegraphics[scale=0.4]{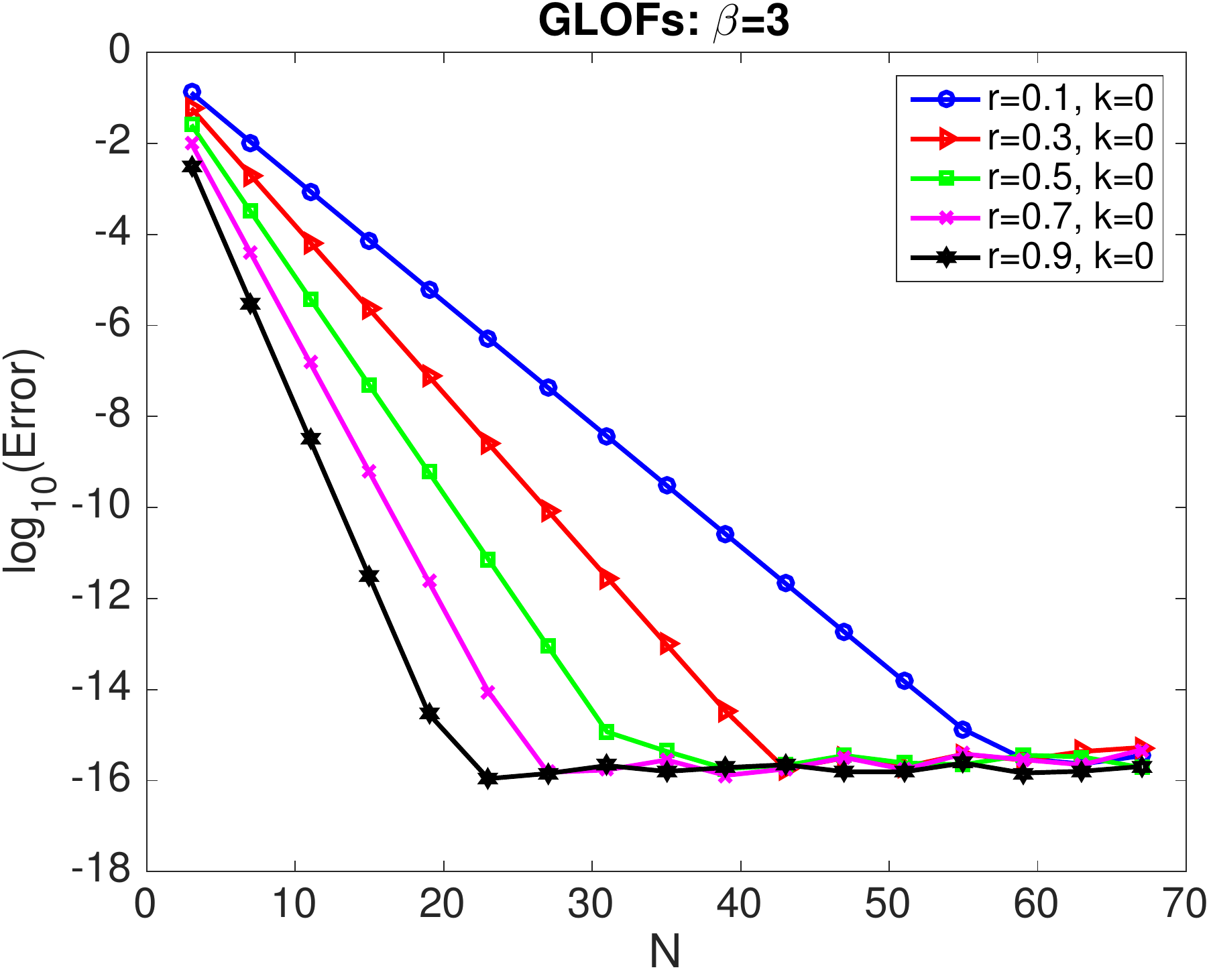}
\end{center}
\end{minipage}
\begin{minipage}{0.495\linewidth}
\begin{center}
\includegraphics[scale=0.4]{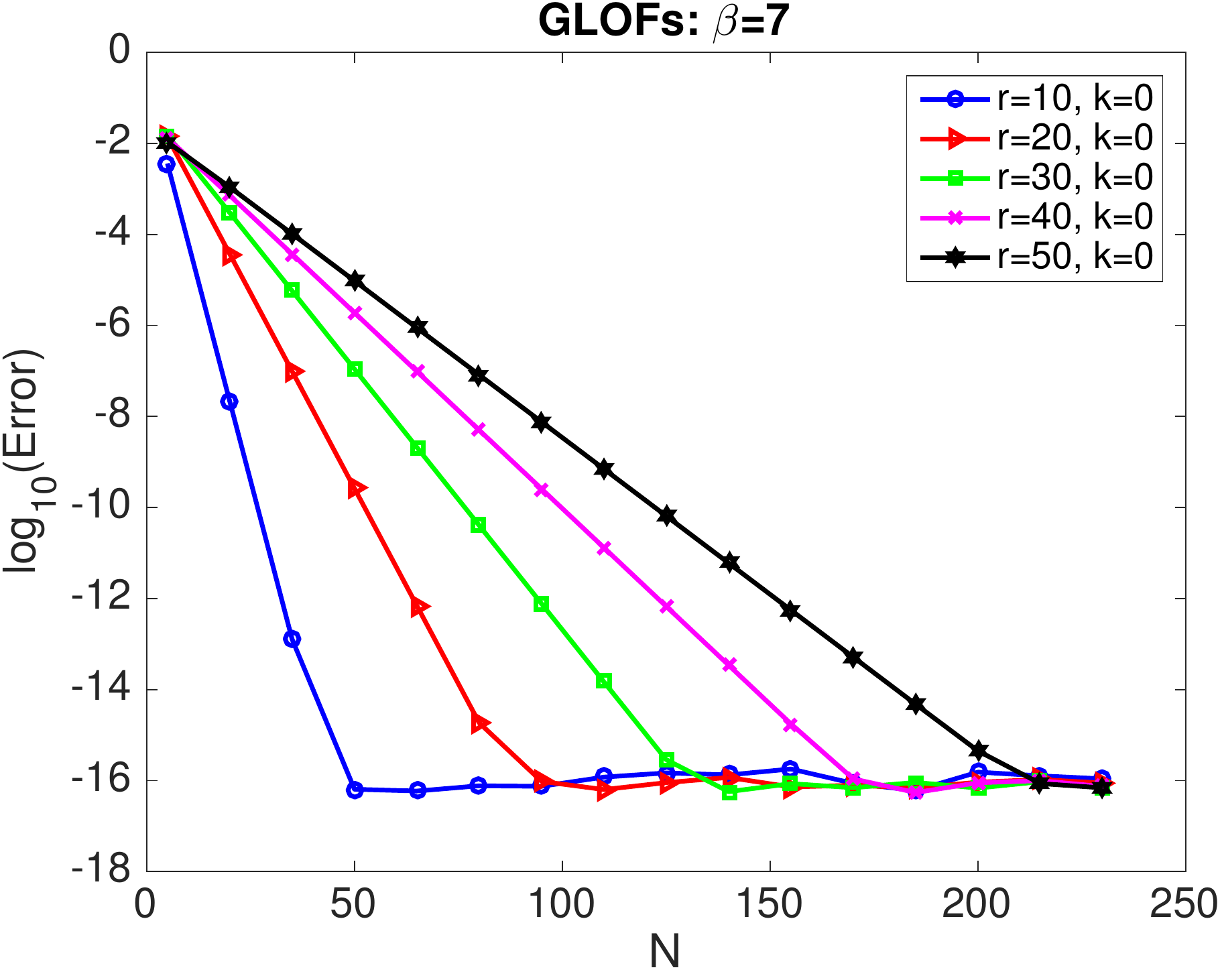}
\end{center}
\end{minipage}
\caption{GLOFs Approximation: $\mathcal{S}^{(\alpha,\beta,\lambda)}_n$, $\alpha=\lambda=0$. }\label{graph_GLOFsApproxi}
\end{figure}

 \begin{figure}[htp!]
\begin{minipage}{0.495\linewidth}
\begin{center}
\includegraphics[scale=0.4]{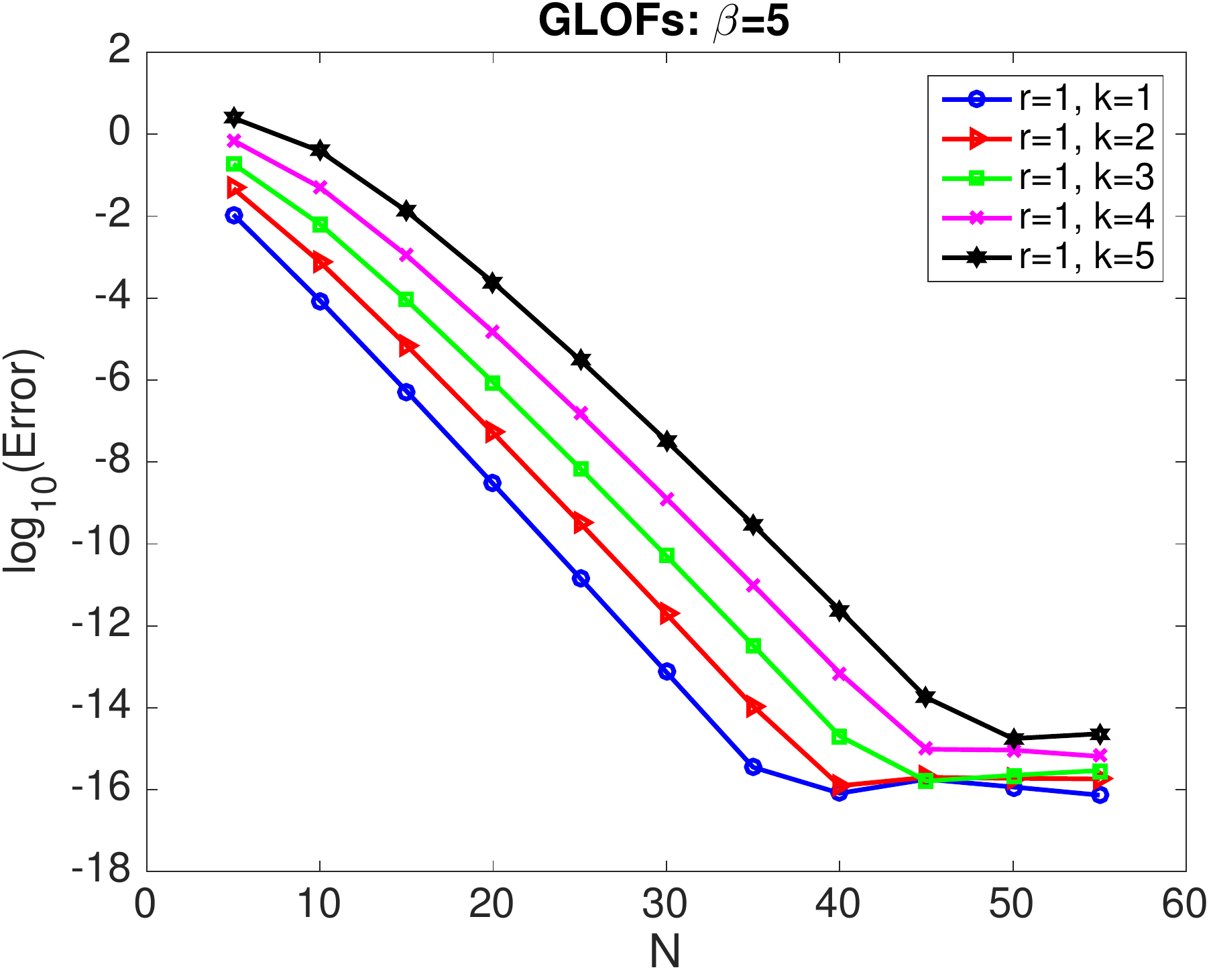}
\end{center}
\end{minipage}
\begin{minipage}{0.495\linewidth}
\begin{center}
\includegraphics[scale=0.4]{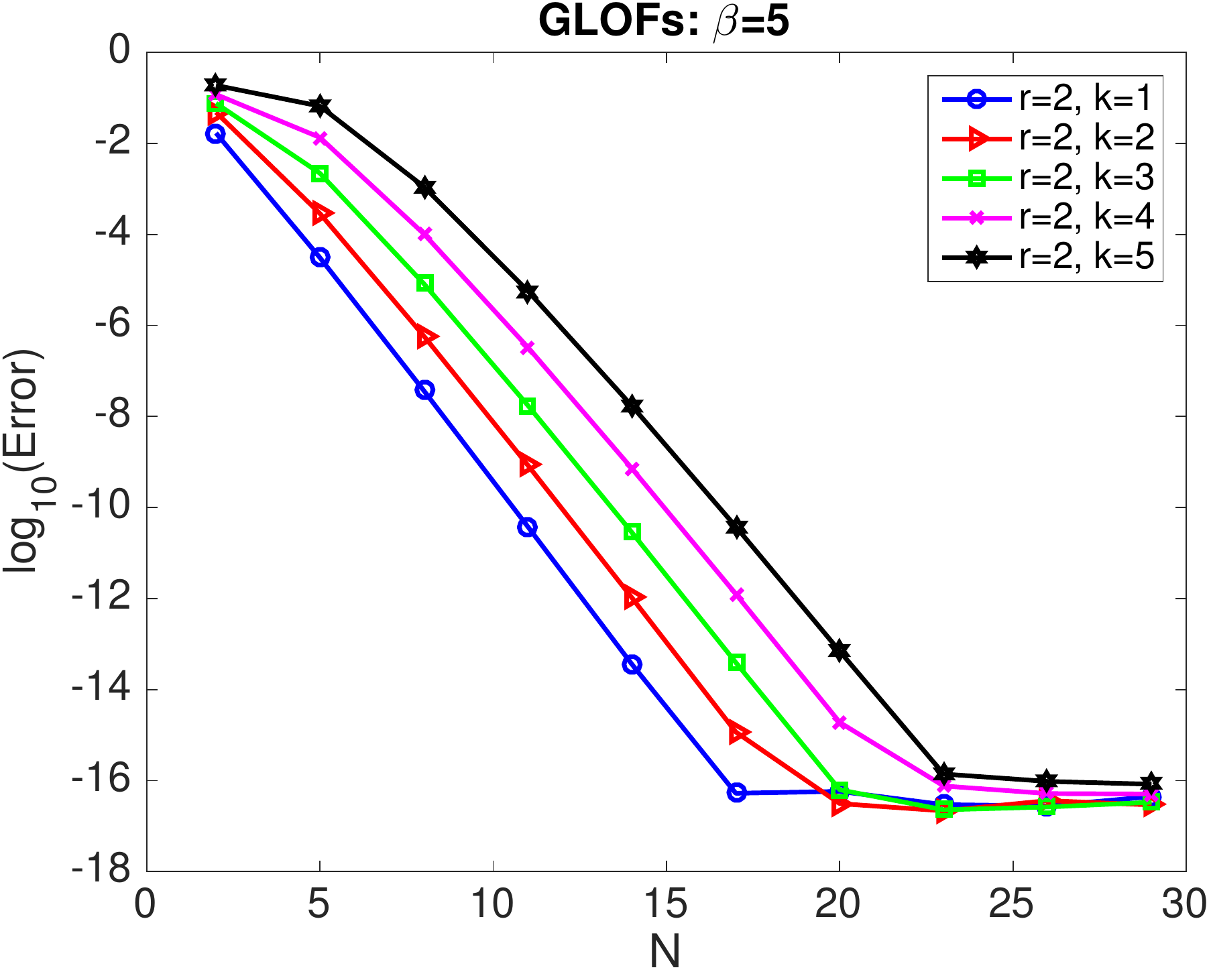}
\end{center}
\end{minipage}
\caption{GLOFs Approximation: $\mathcal{S}^{(\alpha,\beta,\lambda)}_n$, $\alpha=\lambda=0$. }\label{graph_GLOFsApproxi2}
\end{figure}
}

\section{Application to fractional differential equations}
\setcounter{equation}{0}
\setcounter{lmm}{0}
\setcounter{thm}{0}
 In this section, we shall  use GLOFs as the basis functions to solve some typical fractional differential equations. 

We first review the definitions of  Riemann-Liouville and Caputo fractional integrals and fractional derivatives (see e.g., \cite{podlubny1999fractional,samko1993fractional}).
\begin{defn}[{\bf Fractional integrals and derivatives}]\label{RLFDdefn}
 {\em For $t\in{I}=(0,1)$ and $\rho\in {\mathbb R}^+,$  the left and right fractional integrals are respectively defined  as
 \begin{equation}\label{leftintRL}
    \I{0}{t}{\rho} f(t)=\frac 1 {\Gamma(\rho)}\int_{0}^t \frac{f(\tau)}{(t-\tau)^{1-\rho}} {\rm d}y,\quad   \I{t}{1}{\rho} f(t)=\frac {1}  {\Gamma(\rho)}\int_{t}^1 \frac{f(\tau)}{(\tau-t)^{1-\rho}} {\rm d}\tau,.
\end{equation}
For real $s\in [k-1, k)$ with $k\in {\mathbb N},$  the  Riemann-Liouville fractional derivatives are defined by
\begin{equation}\label{RLdefn}
   \D{0}{t}{s} f(t)=\frac{{\rm d}^k}{{\rm d}t^k}\{\I{0}{t}{k-s} f(t)\},\quad    \D{t}{1}{s} f(t)=(-1)^k\frac{{\rm d}^k}{{\rm d}t^k}\{\I{t}{1}{k-s} f(t)\}.
\end{equation}
The Caputo fractional derivative  of order $s$ is defined  by
\begin{equation}\label{Caputodefn}
   \CD{0}{t}{s} f(t)=\I{0}{t}{k-s}\{\frac{{\rm d}^k}{{\rm d}t^k} f(t)\},\quad    \CD{t}{1}{s} f(t)=(-1)^k\I{t}{1}{k-s}\{\frac{{\rm d}^k}{{\rm d}t^k} f(t)\}.
\end{equation}}
\end{defn}


\subsection{An initial value problem (IVP)}
Given $g,\,q\in L^2({I})$, we consider the following Caputo fractional  differential equation of order $\nu\in (0,1)$:
\begin{equation}\label{fde_i0}
\CD{0}{t}{\nu}u(t)+q(t)u(t)=g(t),\quad t\in{I}; \qquad u(0)=u_0.
\end{equation}
We shall first homogenize the initial condition. Setting $u=v+u_0$ into the above equation, we find that the  problem \eqref{fde_i0} is equivalent to
\begin{equation}\label{fde_i}
\CD{0}{t}{\nu}v(t)+q(t)v(t)=g(t)-u_0q(t), \qquad v(0)=0.
\end{equation}
A main difficulty in obtaining accurate approximate solution of   \eqref{fde_i} is that the solution of this problem  is weakly singular at $t=0$ even if   $q$ and  $g$ are smooth. To design an effective approach to deal with this difficulty, we need to understand the nature of this singularity.

 Applying $\I{0}{t}{\nu}$ into both sides of \eqref{fde_i} and using  the fact that $ \I{0}{t}{s} \I{0}{t}{r}= \I{0}{t}{s+r}$, we find
$$
v(t)+\frac{1}{\Gamma(\nu)}\int_{0}^t(t-\tau)^{\nu-1} q(\tau)v(\tau) {\rm d}\tau=\I{0}{t}{\nu}\{g-u_0q\}(t).
$$
We then find from \cite[Theorem 2.1]{cao2003hybrid} that the solution
near $t=0$ behaves like
\begin{equation}\label{singularity_solu}
v(t)=\sum_{i=0}^\infty\sum_{j=1}^\infty \tilde{v}_{ij}~ t^{i+j\nu}.
\end{equation}
This is why usual approximations based on global or piece-wise polynomials
can not  approximate $v(t)$ well. On the other hand, based on the analysis from the last section,  the GLOFs  are particularly suitable for this problem.

 Let us define $X^0_N=\{t^{\frac{\beta-\lambda}{2}}p:~p\in  \mathcal{P}^{\log{t}}_N,~\beta>\lambda\}$. Then, the GLOF-Galerkin method for \eqref{fde_i} is: find $v_N\in X^0_N$ such that
 \begin{equation}\label{fde_i_scheme}
(\CD{0}{t}{\nu}v_N,w)+(qv_N,w)=\big(\mathcal{I}^{\alpha,\beta,\lambda}_N\{g-u_0 q\},w\big),\quad \forall w\in X^0_N.
\end{equation}
Writing $$v_N=\sum\limits_{n=0}^{N}\tilde{v}^{\alpha,\beta,\lambda}_n \mathcal{S}^{(\alpha,\beta,\lambda)}_n, \quad \bar v=( \tilde{v}^{\alpha,\beta,\lambda}_0,\tilde{v}^{\alpha,\beta,\lambda}_1,\cdots,\tilde{v}^{\alpha,\beta,\lambda}_N)^t,$$
and setting
\begin{equation}
 \begin{aligned}
&S_{kj}= (\CD{0}{t}{\nu} \mathcal{S}^{(\alpha,\beta,\lambda)}_j, \mathcal{S}^{(\alpha,\beta,\lambda)}_k), \quad S=(S_{kj}),\\
&M_{kj}=(q \mathcal{S}^{(\alpha,\beta,\lambda)}_j, \mathcal{S}^{(\alpha,\beta,\lambda)}_k), \quad M=(M_{kj}),\\
&f_j=(\mathcal{I}^{\alpha,\beta,\lambda}_N\{g-u_0 q\}, \mathcal{S}^{(\alpha,\beta,\lambda)}_j),\quad \bar f=(f_0,f_1,\ldots,f_N)^t,
\end{aligned}
\end{equation}
then \eqref{fde_i_scheme} reduces to the following linear system
\begin{equation}(S+M) \bar v=\bar f.\label{matrix}\end{equation}
The entries of $M$ and $\bar f$ can be computed accurately by using the
 Gauss-LOFs quadrature formula, but the computation of the stiffness matrix $S$ needs special care.

  Indeed, for any $v,\,w\in X_N^0$,
 $$\big(\CD{0}{t}{\nu}v,w\big)=\int_{0}^1\frac{1}{\Gamma(1-\nu)} \int_{0}^t \frac{v^{\prime}(s)}{(t-s)^{\nu}}\,{\rm d}s~ w(t){\rm d} t\overset{s=t\tau}{=}\frac{1}{\Gamma(1-\nu)} \int_0^1\int_0^1\frac{v^{\prime}(t\tau)}{(1-\tau)^{\nu}}\,{\rm d}\tau\,w(t)t^{1-\nu} {\rm d} t.
$$
Note that the integrand in the above is weakly singular as   $t\rightarrow 0$ and $\tau\rightarrow 1$. In order to compute accurately the inner  integral, we split it into two terms
\begin{equation*}\begin{aligned}
 \int_0^1&\frac{v^{\prime}(t\tau)}{(1-\tau)^{\nu}}\,{\rm d}\tau=\int_0^\frac{1}{2}v^{\prime}(t\tau)\,(1-\tau)^{-\nu}\,{\rm d}\tau+\int_\frac{1}{2}^1v^{\prime}(t\tau)\,(1-\tau)^{-\nu}\,{\rm d}\tau\\
 &=\frac{1}{2}\int_0^1v^{\prime}(\frac{t\tau}{2})\,(1-\frac{\tau}{2})^{-\nu}\,{\rm d}\tau+\frac{1}{4^{1-\nu}}\int_{-1}^1v^{\prime}_N(\frac{t(\xi+3)}{4})\,(1-\xi)^{-\nu}\,{\rm d}\xi.
\end{aligned}\end{equation*}
Hence,
\begin{equation}
\begin{split}
 \big(\CD{0}{t}{\nu}v,w\big)&=\frac{1}{2\Gamma(1-\nu)} \int_0^1\int_0^1v^{\prime}(\frac{t\tau}{2})\,(1-\frac{\tau}{2})^{-\nu}\,{\rm d}\tau\, w(t)t^{1-\nu} {\rm d} t\\
 &+\frac{1}{4^{1-\nu}\Gamma(1-\nu)} \int_0^1 \int_{-1}^1v^{\prime}(\frac{t(\xi+3)}{4})\,(1-\xi)^{-\nu}\,{\rm d}\xi \, w(t)t^{1-\nu} {\rm d} t.
\end{split}
\end{equation}
The first term has weak singularity as $t,\tau\rightarrow 0$ while the second term
has  weak singularities as $t\rightarrow 0$ and $\xi\rightarrow 1$.
Therefore, the first term can be computed by using the tensor product of Gauss-GLOF quadratures (in $t$ and $\tau$)  which is effective with weak singularities as $t,\tau\rightarrow 0$, and the second term can be computed by using the tensor product of Gauss-GLOF quadrature in $t$ and of Gauss-Jacobi quadrature in $\xi$ with weight function $(1-\xi)^{-\nu}$. More precisely,
\begin{equation*}\begin{aligned}
(\CD{0}{t}{\nu}v_N,w)\approx&\frac{1}{2\Gamma(1-\nu)}\sum_{i=0}^{N_I}\sum_{j=0}^{N_I} v^{\prime}_N(\frac{t_i~t_j}{2})\,(1-\frac{t_j}{2})^{-\nu}~w(t_i)t_i^{1-\nu}~ \chi_i\chi_j
\\&
+\frac{1}{4^{1-\nu} \Gamma(1-\nu)}\sum_{i=0}^{N_I}\sum_{j=0}^{N_I} v^{\prime}_{N}(\frac{{t_i}(\xi_j+3)}{4})\,w(t_i)t_i^{1-\nu}~ \chi_i\eta_j
\end{aligned}\end{equation*}
where $N_I\ge N$ is a suitable number,   $\{t_i,\chi_i\}_{i=0}^{N_I}$ are the  Gauss-LOFs  nodes with weight function $\chi^{0,0}\equiv 1$, and $\{\xi_i,\eta_i\}_{i=0}^{N_I}$  are the  Gauss-Jacobi nodes  with weight function $(1-\tau)^{-\nu}$.


We  present below some numerical results.  We  consider
 \begin{equation}\label{MLF}
\CD{0}{t}{\nu}u(t)+\mathcal{K} u(t)=0,\qquad u(0)=1,
\end{equation}
whose solution  is \cite[Theorem 4.3]{diet10}  $u(t)=E_{\nu}(-\mathcal{K} t^\nu)$  where  $E_{\gamma}(z)$ is
the Mittag-Leffler function
\begin{equation}\label{M_L}
E_{\gamma}(z)=\sum_{j=0}^\infty \frac{z^j}{\Gamma(\gamma j+1)}.
\end{equation}

We fix the parameters $\alpha=0,\beta=5,\lambda=0$, and plot in the left of Fig. \ref{fig_sec5_2} the convergence rates for various values of $\nu$ with $\mathcal{K}=1$. It is clear that the solution is not smooth in the classical Sobolev space, but it is smooth in the space  defined through the pseudo-derivative, so we still obtain an exponential convergence rate.

\begin{figure}[htp!]
\begin{minipage}{0.495\linewidth}
\begin{center}
\includegraphics[scale=0.4]{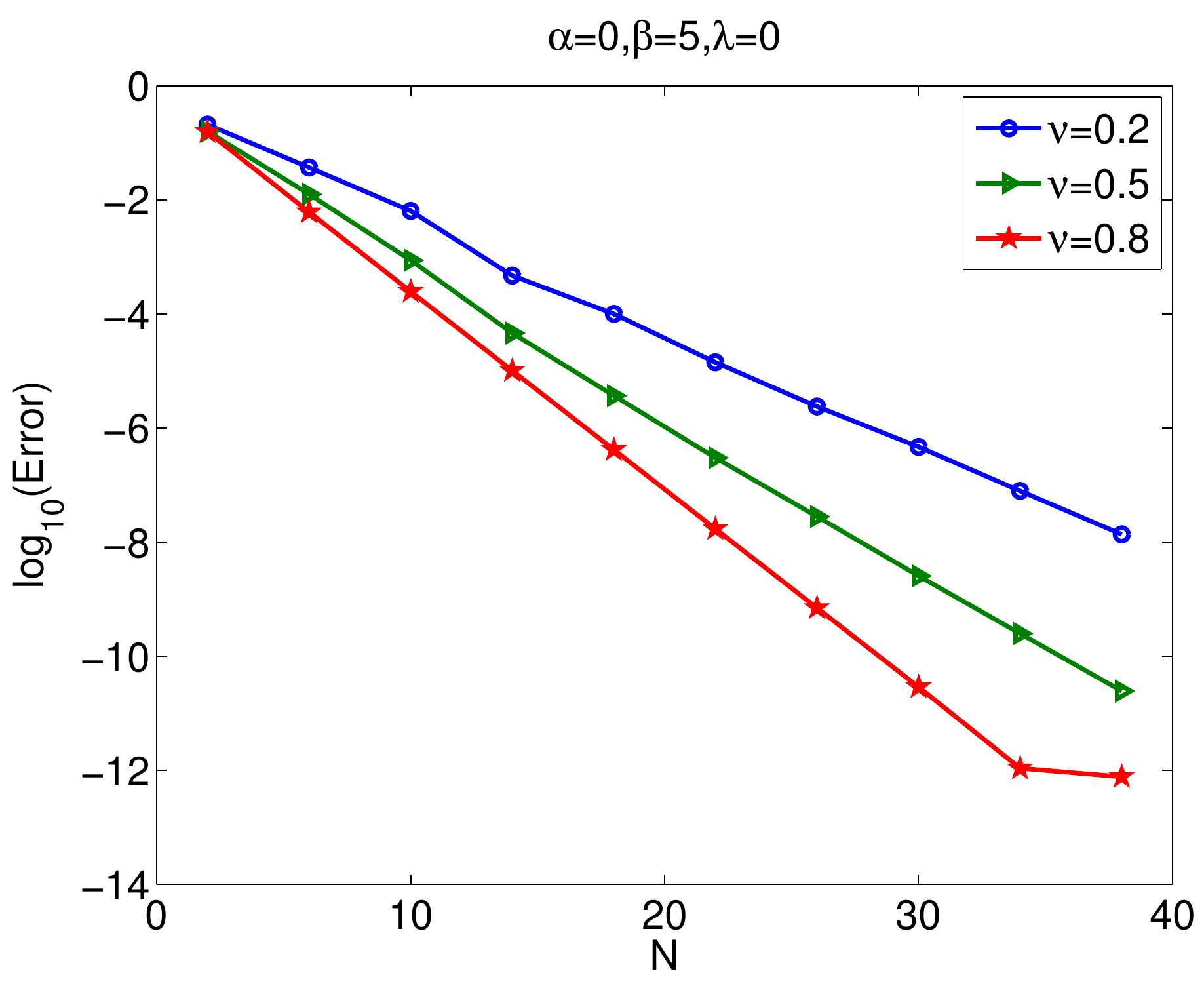}
\end{center}
\end{minipage}
\begin{minipage}{0.495\linewidth}
\begin{center}
\includegraphics[scale=0.4]{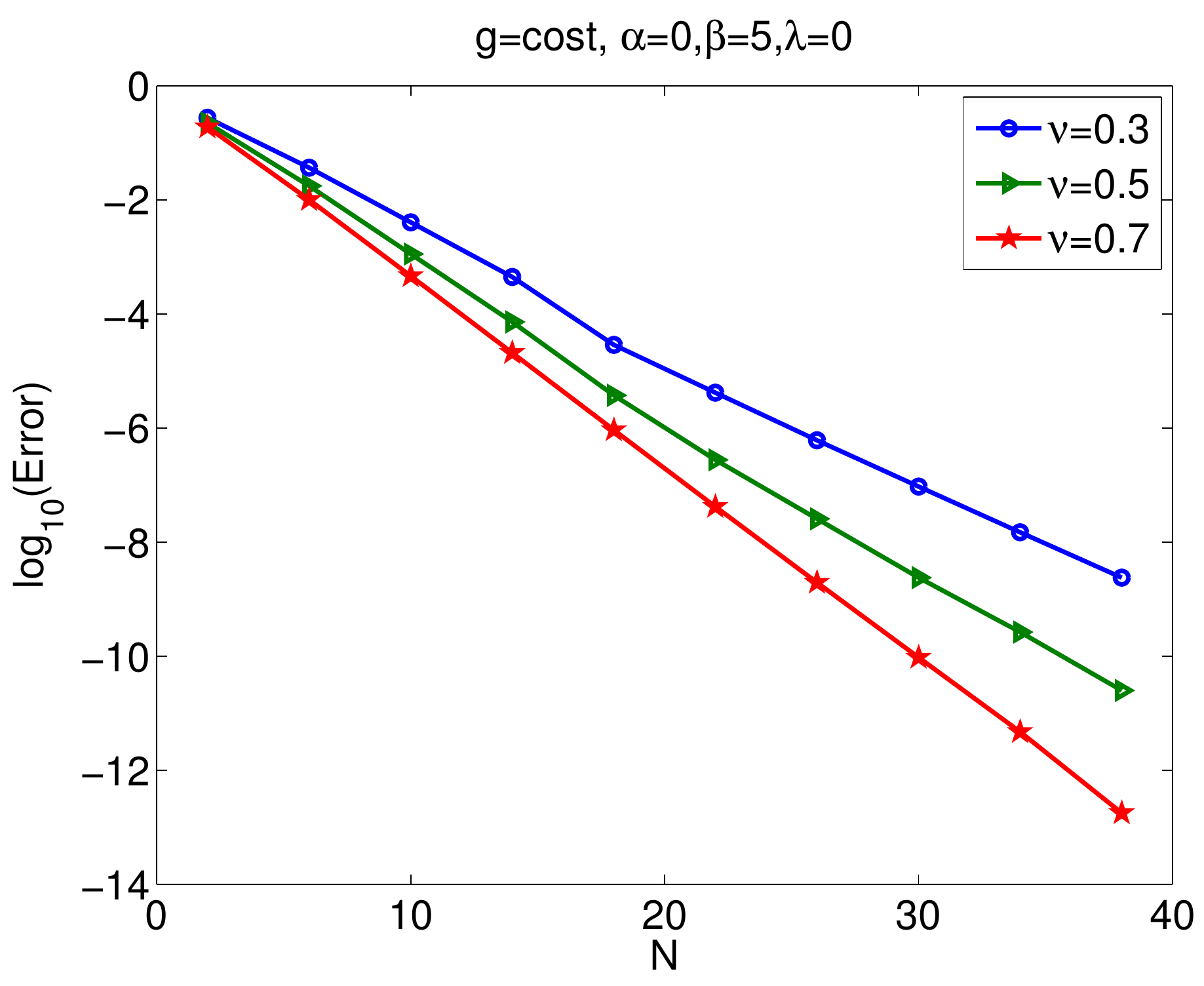}
\end{center}
\end{minipage}
\caption{ Left: \eqref{MLF} with $\mathcal{K}=1$. \quad Right:  $g=\cos{t}$.}\label{fig_sec5_2}
\end{figure}

Next, we consider
\begin{equation*}
\CD{0}{t}{\nu}u(t)+(1+\sin{t})u(t)=\cos t,\qquad u(0)=1,
\end{equation*}
for which the exact solution is unknown. Although the coefficients are smooth, the exact solution  is expected to be  weakly singular near zero but smooth in the space  defined through the pseudo-derivative.
We fix the parameters $\alpha=0,\beta=5,\lambda=0$, and plot in the right of Fig.~\ref{fig_sec5_2} the convergence rates for various values of $\nu$. We  obtain again  exponential convergence rates.

\subsection{A boundary value problem (BVP)}

We consider
\begin{equation}\label{fde_b0}
\begin{cases}
-\D{0}{t}{\mu}u(t)+q(t)u(t)=g(t),\quad t\in{I},\\
 u(0)=0,\quad u(1)=0,
 \end{cases}
\end{equation}
where $\mu\in (1,2)$ and  $g,\,q$ are given functions.

 Similar  to the initial problem \eqref{fde_i0},  the solution of the above problem is usually weakly singular even with smooth $g$ and $q$. However, it can be approximated accurately by GLOFs since the solution is smooth in the space defined through  the pseudo-derivative  \eqref{pseudo2}.

 Let us denote
 $$X^{0,0}_N:=\text{span}\{\phi_n=\frac{n}{n+\alpha}\mathcal{S}^{(\alpha,\beta,\lambda)}_n-\mathcal{S}^{(\alpha,\beta,\lambda)}_{n-1}:~n=1,2,\ldots,N,~\beta>\lambda\}.$$
Note that we have $\phi_n(0)=\phi_n(1)=0$ for $n\ge 1$. 
Our  GLOF Galerkin method is: find
 $u_N\in X^{0,0}_N$
 such that
 \begin{equation}\label{fde_b_scheme}
-(\D{0}{t}{\mu}u_N,w)+(qu_N,w)=(\mathcal{I}^{\alpha,\beta,\lambda}_Ng,w),\quad \forall w\in X^{0,0}_N.
\end{equation}
The stiffness and mass matrices of the above problem can be formulated as in the case of the initial value problem considered above.

We now present some numerical results. We first take $q(t)=e^t$ and the exact solution to be $u(t)=t^{3/2}(1-t)$. The convergence rate is shown
on the left of the Fig. \ref{fig_sec5_3}. We then take $q(t)=e^t$ and $g(t)=t\sin{t}$. In this case, the exact solution is not known explicitly so we used a very fine mesh to compute a reference solution.
The convergence rate is shown
on the right of the Fig. \ref{fig_sec5_3}. We observe that the error converges exponentially in both cases despite the fact that the solutions are weakly singular near $t=0$.

\begin{figure}[htp!]
\begin{minipage}{0.495\linewidth}
\begin{center}
\includegraphics[scale=0.4]{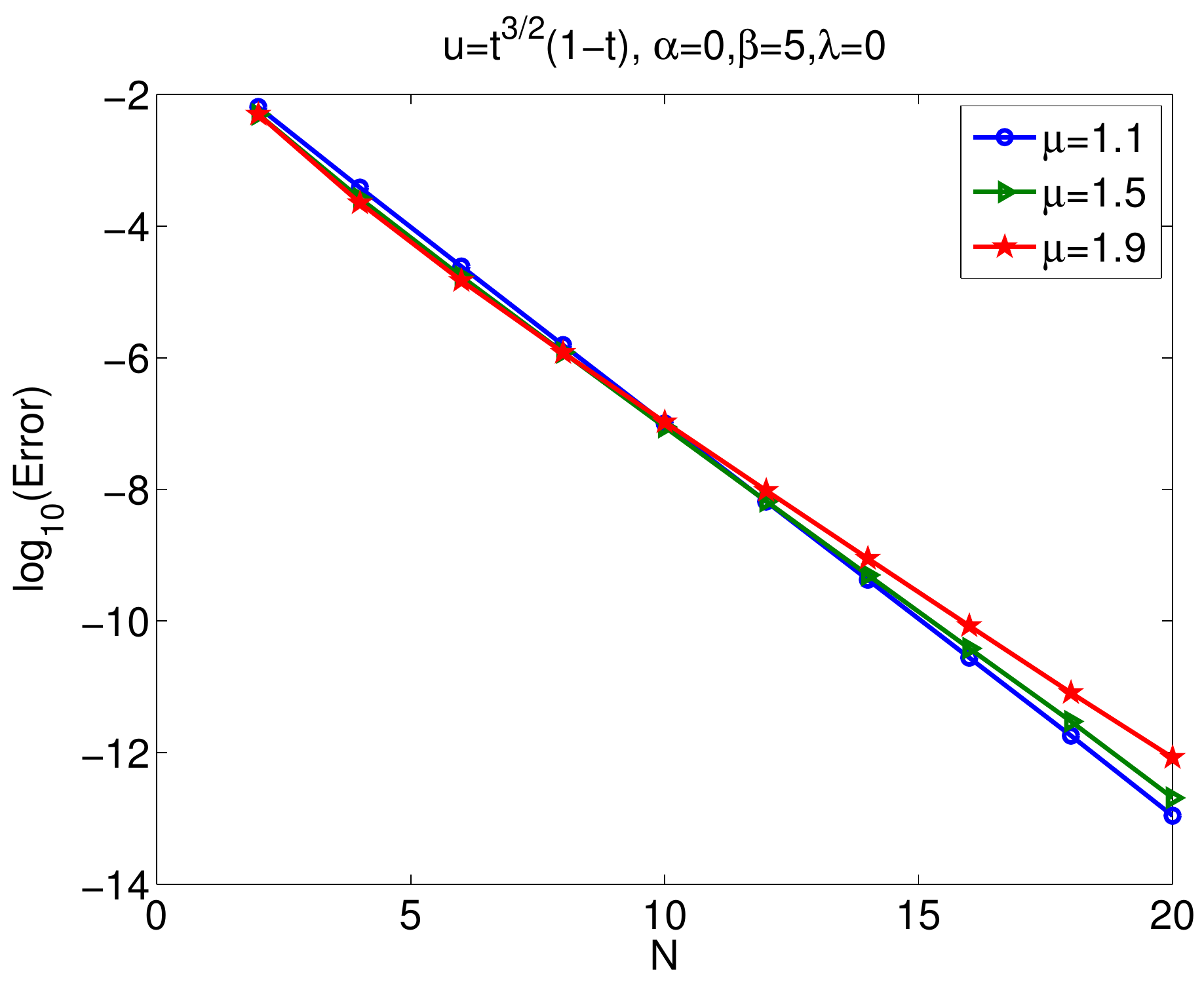}
\end{center}
\end{minipage}
\begin{minipage}{0.495\linewidth}
\begin{center}
\includegraphics[scale=0.4]{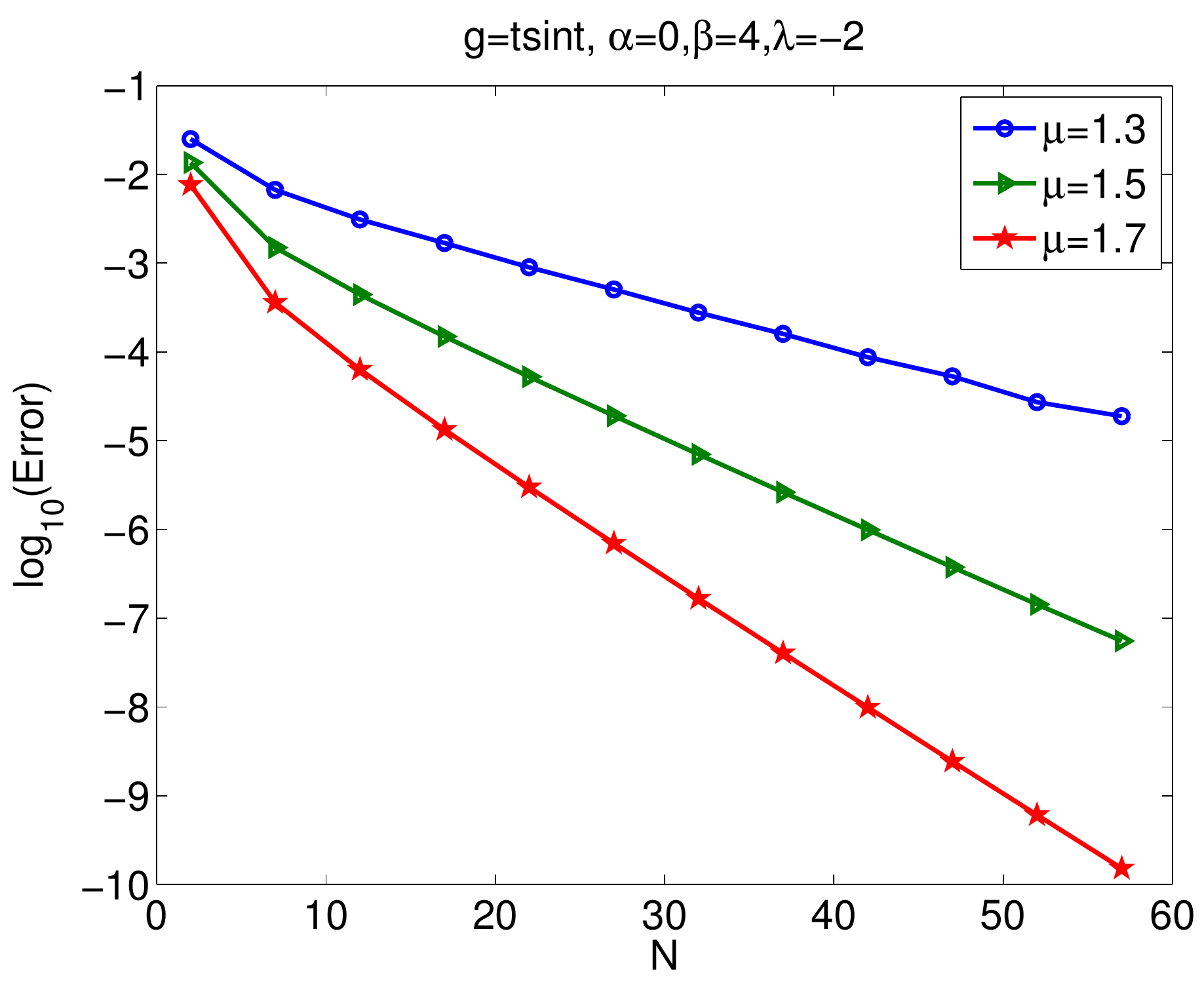}
\end{center}
\end{minipage}
\caption{ Left: $q(t)=e^t$ and $u=t^{3/2}(1-t)$. \quad Right: $q(t)=e^t$ and $g=t\sin{t}$.}\label{fig_sec5_3}
\end{figure}

\subsection{Error analysis}
We carry out below error analysis for the GLOF Galerkin schemes for both the initial and boundary value problems.

We define
\begin{itemize}
\item{} for $0\leq s<\frac{1}{2}$,
\begin{equation*}
H^s_0({I}):=\{f\in L^2({I}):~\D{0}{t}{s}f,~\D{t}{1}{s}f \in L^2({I})\};
\end{equation*}
\item{} for $\frac{1}{2}<s\leq{1}$,
\begin{equation*}
H^s_0({I}):=\{f\in L^2({I}):~\D{0}{t}{s}f,~\D{t}{1}{s}f \in L^2({I}),~f(0)=f(1)=0\}.
\end{equation*}
\end{itemize}
equipped with norm and semi-norm
\begin{equation*}
\|f\|_s=\sqrt{\|f\|^2+|f|_s^2} \quad\text{with}\quad |f|_s=\sqrt{\|\D{0}{t}{s}f\|^2+\|\D{t}{1}{s}f\|^2}.
\end{equation*}
It can be shown that the above definitions coincide with the usual definitions by space interpolation.

To avoid repetition, we use the following
  weak formulation for both problems \eqref{fde_i} and \eqref{fde_b0}: Find $P\in H^s_0({I})$ such that
\begin{equation}\label{weak_both}
a(P,w):=(-1)^{\sigma(s)} (\D{0}{t}{s}P,\D{t}{1}{s}w)+(qP,w)=(Q,w),\quad \forall w\in H^s_0({I}),
\end{equation}
where
\begin{itemize}
\item{} for  \eqref{fde_i}:\qquad $s={\nu}/{2}$,\quad$\sigma(s)=0$,\quad $P(t)=v(t)$,  \quad$Q(t)=\{g-u_0q\}(t)$
\item{} for  \eqref{fde_b0}:\,\,\,\quad$s={\mu}/{2}$,\quad$\sigma(s)=1$,\quad $P(t)=u(t)$, \quad $Q(t)=g(t)$
\end{itemize}

The error analysis follows similar procedures used in  \cite{ervin2006variational} and \cite{Li.X09,Li.X10}. We first recall some useful results.
\begin{lmm}Let $s,r\in [0,1]/\{\frac{1}{2}\}$ and $s\leq r$. For any $f,h\in H^{r}_0({I})$, there exists
\begin{itemize}
\item {\bf \cite[Corollary 2.15]{ervin2006variational}}
\begin{equation}\label{Poincare}
\|f\|\leq c_1|f|_s\leq c_2 |f|_r.
\end{equation}
\item \cite[Lemma 2.8]{Li.X10}
 \begin{equation}\label{lmm11}
(\D{0}{t}{2s}f,h)=(\D{0}{t}{s}f,\D{t}{1}{s}h).
\end{equation}
\item \cite[Lemma 2.6]{Li.X10}
\begin{equation}\label{lmm12}
c_1 |f|_s^2\leq \frac{(\D{0}{t}{s}f,\D{t}{1}{s}f)}{\cos(s\pi)}\leq c_2 |f|_s^2.
\end{equation}

\end{itemize}
where $c_1,~c_2$ are two positive constants independent of function $f$.
\end{lmm}

Thanks to relation \eqref{lmm11}  and the identity  below
$$(\D{0}{t}{\mu}u,w)=(\D{0}{t}{s}u,\D{t}{1}{s} w),\quad \mu=2s,~s\in(1/2,1),$$
 we can rewrite  \eqref{fde_i_scheme} and \eqref{fde_b_scheme}  as: find $P_N\in X_N$   such that
\begin{equation}\label{scheme_both}
a(P_N,w)=(\mathcal{I}^{\alpha,\beta,\lambda}_NQ,w),\quad \forall w\in X_N,
\end{equation}
where $X_N=X^0_N$ for \eqref{fde_i_scheme}, and  $X_N=X^{0,0}_N$ for \eqref{fde_b_scheme}.

\begin{lmm}\label{lmm_elliptic}
 If $\min\limits_{t\in{[0,1]}}q(t)\geq 0$, then there exists $c(s)>0$ such that for any $P,\tilde{P}\in H^s_0({I})$, we have
\begin{equation}\label{elliptic}
c(s)\|P\|_s^2\leq  a(P,P),\quad a(P,\tilde{P})\leq \|P\|_s\|\tilde{P}\|_s.
\end{equation}
\begin{proof}
Due to the fact $(-1)^{\sigma(s)} \cos(s\pi)>0$ and the relation \eqref{lmm12}, it is easy to derive that
$$(-1)^{\sigma(s)} \cos(s\pi)|P|_s^2\leq (-1)^{\sigma(s)}{(\D{0}{t}{s}P,\D{t}{1}{s}P)} \leq  a(P,P).$$
We can then derive the first inequality in \eqref{elliptic} from the generalized  Poincare inequality \eqref{Poincare}. The second inequality in \eqref{elliptic} is a direct consequence of Cauchy Schwarz inequality.
\end{proof}
\end{lmm}

Thanks to \eqref{elliptic}, the existence-uniqueness  of the weak formulation \eqref{weak_both} and  the schemes \eqref{scheme_both} follows immediately from the Lax-Milgram Lemma.

As for the error estimate, we have the following result.

\begin{thm}
Let $-1<\alpha\leq 0$, $\lambda\leq 0$ and $\beta>1$. Let $P$ and $P_N$ be respectively  the solution of  \eqref{weak_both} and \eqref{scheme_both} with  $\min\limits_{t\in{[0,1]}}q(t)\geq 0$.
Then, we have
\begin{equation}\begin{aligned}\label{thm_final}
\|P-P_N\|_s\leq &c \sqrt{(\beta-1)^{-\widetilde{m}}\frac{(N-\widetilde{m}+1)!}{(N+1)!}}~ \|\widehat{\partial}_{\frac{\beta-\lambda-2}{2},t}^{\,\widetilde{m}+1}
 P\|_{\chi^{\alpha+\widetilde{m},\lambda}}
 \\+&c\sqrt{\frac{(N+1-\widetilde{m})!}{(\beta+1)^{\widetilde{m}-\alpha}N!}}\left\{c^{\beta}_1\|\widehat{\partial}_{\frac{\beta-\lambda}{2},t}^{\,\widetilde{m}}{Q}\|_{\chi^{\alpha+m-1,\lambda}}+c^{\beta}_2\sqrt{\log N} \|\widehat{\partial}_{\frac{\beta-\lambda}{2},t}^{\,\widetilde{m}}{Q}\|_{\chi^{\alpha+m,\lambda}}\right\}
\end{aligned}\end{equation}
where $c^{\beta}_1={(\beta+1)^{-\frac{1}{2}}},\quad c^{\beta}_2=2\sqrt{\max\{1,\beta+1\}}$ and $\widetilde{m}=\min\{m,N+1\}$.

\begin{proof}
For any $w_N\in X_N$, we derive from \eqref{weak_both} and \eqref{scheme_both} that
 \begin{equation}\label{eq_uf}
 a(P-P_N,w_N)=(Q-\mathcal{I}^{\alpha,\beta,\lambda}_NQ,w_N).
 \end{equation}
Let $e_N=P_N-w_N$, we have
\begin{equation}\label{eq_uf2}
c(s)\|e_N\|^2_s\leq a(e_N,e_N)=a(P-w_N, e_N)+a(P_N-P, e_N).
\end{equation}
Take $w_N=e_N$ in \eqref{eq_uf}, we find
 \begin{equation}\label{eq_uf3}
a(P_N-P,e_N)=(\mathcal{I}^{\alpha,\beta,\lambda}_NQ-Q, e_N)\leq \|\mathcal{I}^{\alpha,\beta,\lambda}_NQ-Q\|\|e_N\|.
\end{equation}
We then derive from \eqref{eq_uf2} and \eqref{elliptic} that
\begin{equation*}
c(s)\|w_N-P_N\|_s\leq c( \|P-w_N\|_s+\|\mathcal{I}^{\alpha,\beta,\lambda}_NQ-Q\|),
\end{equation*}
which, along with \eqref{eq_uf3}, implies that
\begin{equation}\label{optimal}
\|P-P_N\|_s\leq \|P-w_N\|_s+\|w_N-P_N\|_s \leq c\inf_{w\in X_N }\|P-w_N\|_s+c\|\mathcal{I}^{\alpha,\beta,\lambda}_NQ-Q\|.
\end{equation}
Next, we set
 $$w_N(t)=-\int_t^1 \pi^{\alpha,\beta-2,\lambda}_N\{\partial_\tau P\}(\tau) {\rm d}\tau.$$
Obviously, $w_N(1)=0$. Writing $\pi^{\alpha,\beta-2,\lambda}_N\{\partial_\tau P\}(\tau)=\tau^{\frac{\beta-\lambda}{2}-1}\sum_{k=0}^N \tilde c_k (\log \tau)^k$, and integrating by parts, we find
\begin{equation*}
\int_t^1 \pi^{\alpha,\beta-2,\lambda}_N\{\partial_\tau P\}(\tau) {\rm d}\tau=\int_t^1 \tau^{\frac{\beta-\lambda}{2}-1} \sum_{k=0}^N \tilde c_k (\log{\tau})^k {\rm d}\tau=t^\frac{\beta-\lambda}{2}\sum_{k=0}^N \tilde d_k (\log{t})^k,
\end{equation*}
which implies $w_N(0)=0$. Hence, $w_N\in X_N$.
On the other hand, we have
\begin{equation*}
\partial_tP-\partial_tw_N=(I-\pi^{\alpha,\beta-2,\lambda}_N)\partial_tP.
\end{equation*}
We can then derive from the above relations and
 Poincare inequality \eqref{Poincare}   that
\begin{equation*}\label{thm_eq2}
 \|P-w_N\|_s\le c\|\partial_tP-\partial_tw_N\|\leq c \|\partial_tP-\partial_tw_N\|_{\chi^{\alpha,\lambda}}=c\|\partial_tP-\pi^{\alpha,\beta-2,\lambda}_N\partial_t P\|_{\chi^{\alpha,\lambda}}.
\end{equation*}
Finally, combing the above and \eqref{optimal}  and applying Theorems \ref{thm_proj2} and  \ref{thm_Interpolation2}, we obtain the desired result.
\end{proof}
\end{thm}

\begin{rem}
 As in Corollaries 2.1 and 3.1, we can show that for $P=t^r$ and $Q=t^q$, the
 estimate \eqref{thm_final} leads to exponential convergence rate if $r$ and $q$ are within certain range.
\end{rem}

\subsection{Time-fractional diffusion equations}
As a final application, we consider the time-fractional diffusion equation
\begin{equation}\label{subdiffusion}
\CD{0}{t}{\nu} u(x,t)-\Delta u(x,t)=f(x,t),  \quad x\in\Omega,~t\in(0,T),
\end{equation}
where $\Omega$ be a bounded domain in $\mathbb{R}^d$ ($d=1,2,3$) with suitable initial and boundary conditions. It is clear that the solution of the above equation will exhibit weak singularities at $t=0$ so it is appropriate to use GLOFs for  the time variable. As for the space variables, any consistent approximation can be used. The resulting linear system can be efficiently solved by using a matrix-diagonalization method \cite{Shen94b, ShenTangWang2011}.

As a specific example, we consider $\Omega=(-1,1)^2$ with the following
initial and boundary conditions:
\begin{align}
\label{boundary}&u(x,t)=0, \quad x\in\partial \Omega,~t\in(0,T),\\
\label{initial}&u(x,0)=0, \quad x\in \Omega,
\end{align}
and we use a Legendre-Galerkin method \cite{Shen94b} for the space variables.

Let $N_t$, $N_x$ be respectively  the degree of freedom of GLOFs in time and Legendre polynomials in each spatial direction.

In the first test, we choose the exact solution to be
$u=(t^{\mu}+t^{2\mu})\sin(\pi x_1)\sin(\pi x_2)$ which is smooth in space but has typical weak singularity in time.
The errors in $L^2$-norm with different $\nu$ are plotted in Fig. \ref{graph_SM1}. We observe that the errors converge exponentially w.r.t. $N_t$ and $N_x$. 

In the second test, we take $f=e^{x_1x_2t}$. The explicit form of the exact solution is unknown but is expected to be weakly singular at t=0. We used a fine mesh to compute a reference solution, and plotted the convergence rate in Fig. \ref{graph_SM2}. Again, exponential convergence rates are observed for both $N_t$ and $N_x$. 

 \begin{figure}[htp!]
\begin{minipage}{0.495\linewidth}
\begin{center}
\includegraphics[scale=0.4]{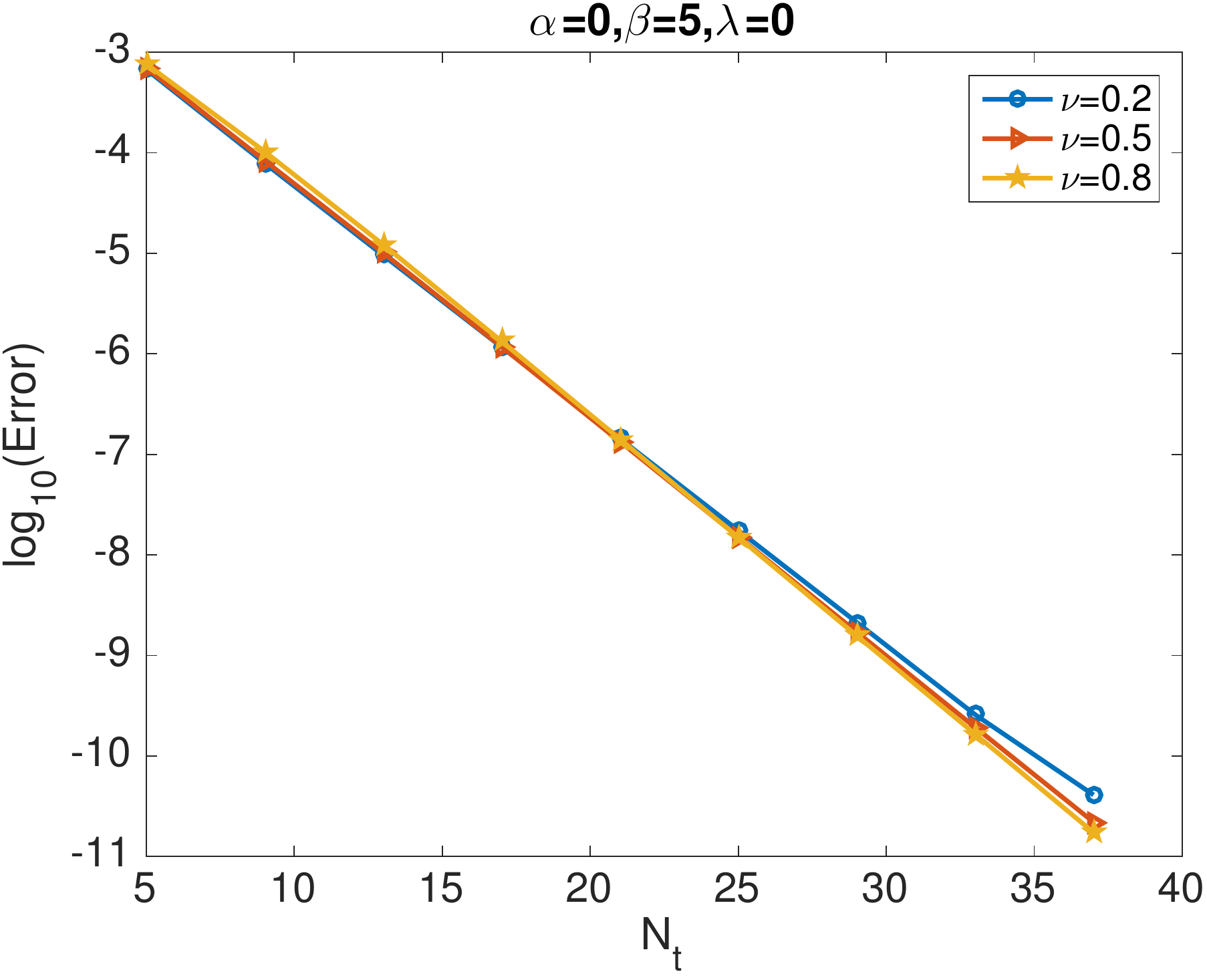}
\end{center}
\end{minipage}
\begin{minipage}{0.495\linewidth}
\begin{center}
\includegraphics[scale=0.4]{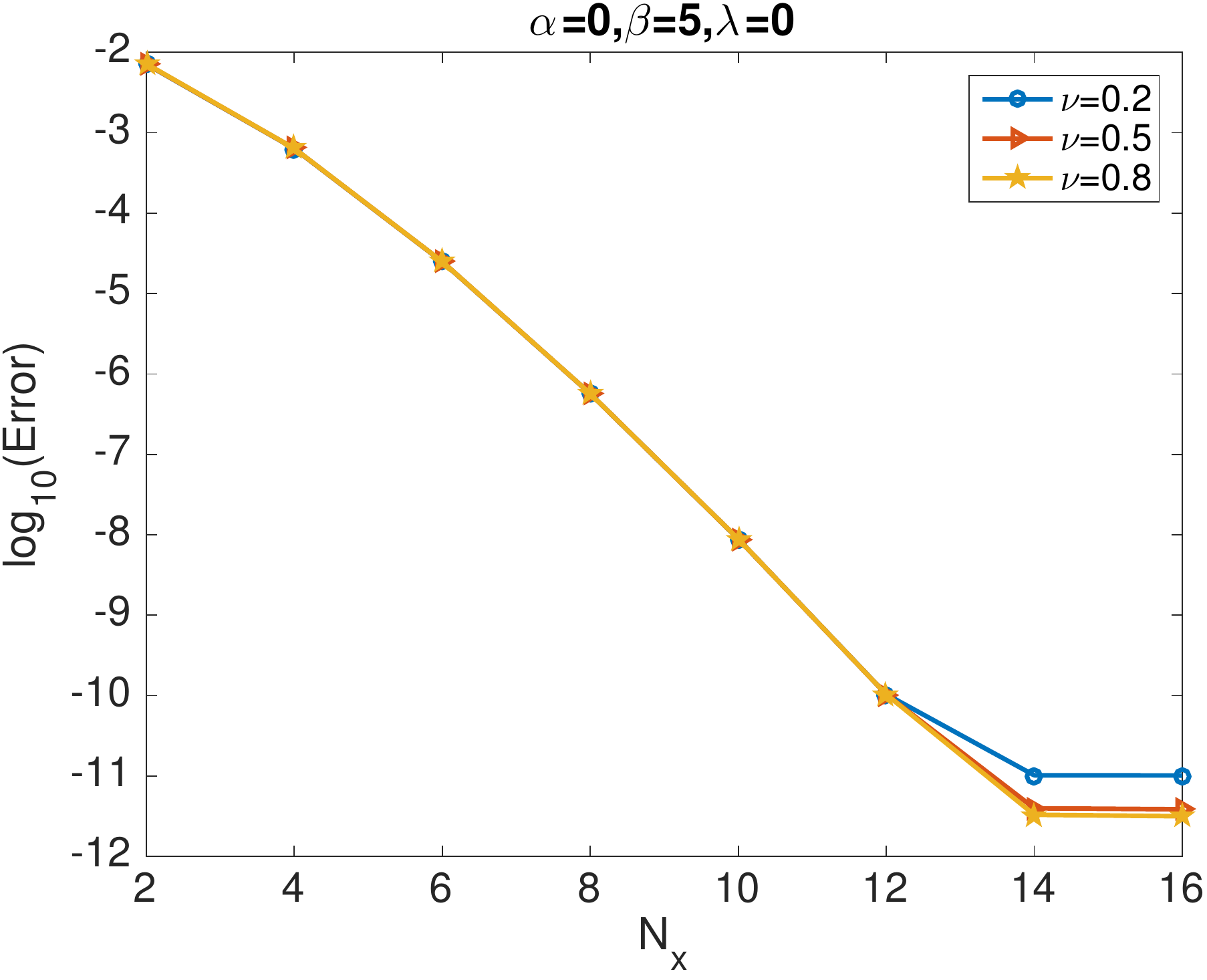}
\end{center}
\end{minipage}
\caption{  $u=(t^{\mu}+t^{2\mu})\sin(\pi x_1)\sin(\pi x_2),~\mu=0.6,~T=1.$ }\label{graph_SM1}
\end{figure}

 \begin{figure}[htp!]
\begin{minipage}{0.495\linewidth}
\begin{center}
\includegraphics[scale=0.4]{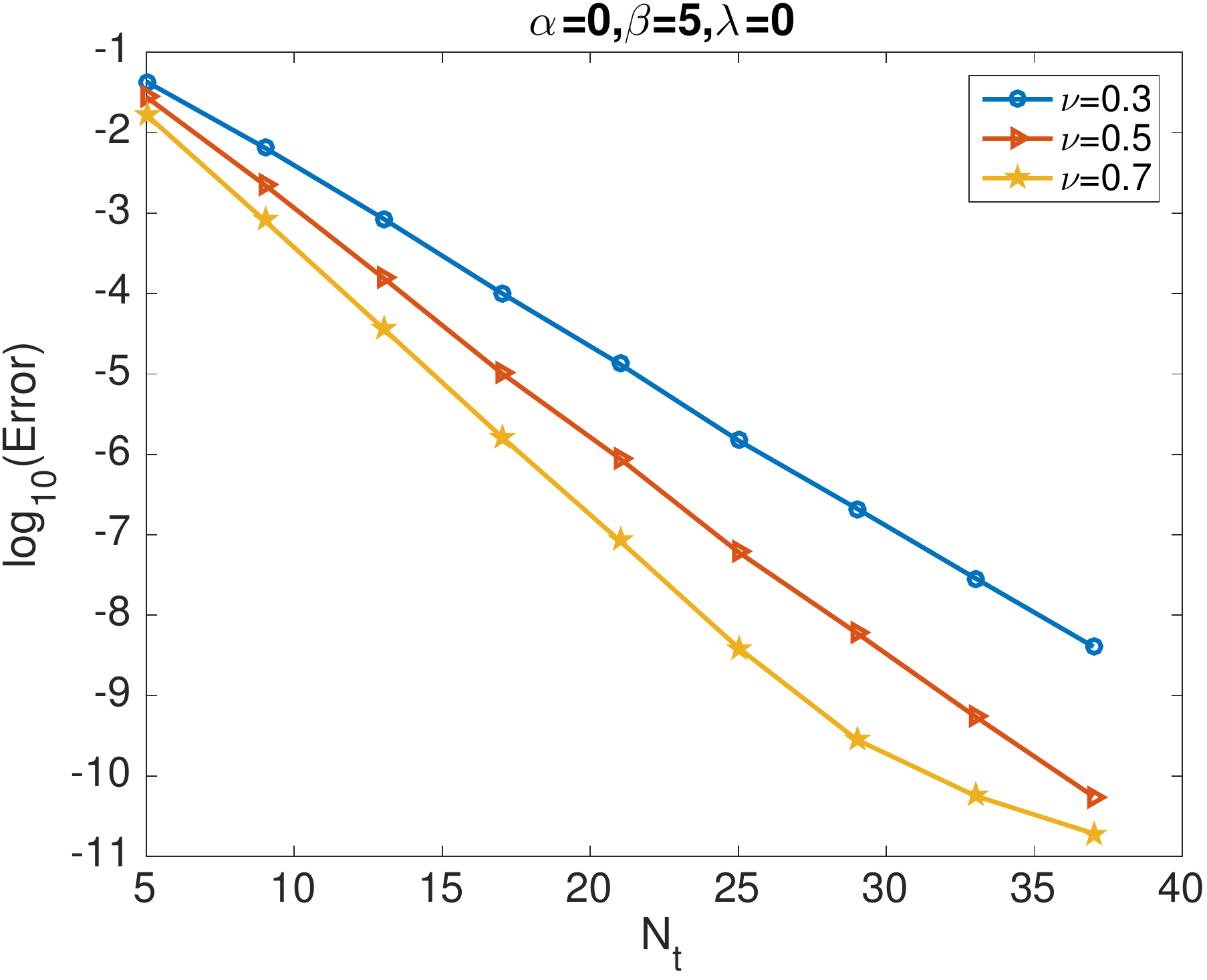}
\end{center}
\end{minipage}
\begin{minipage}{0.495\linewidth}
\begin{center}
\includegraphics[scale=0.4]{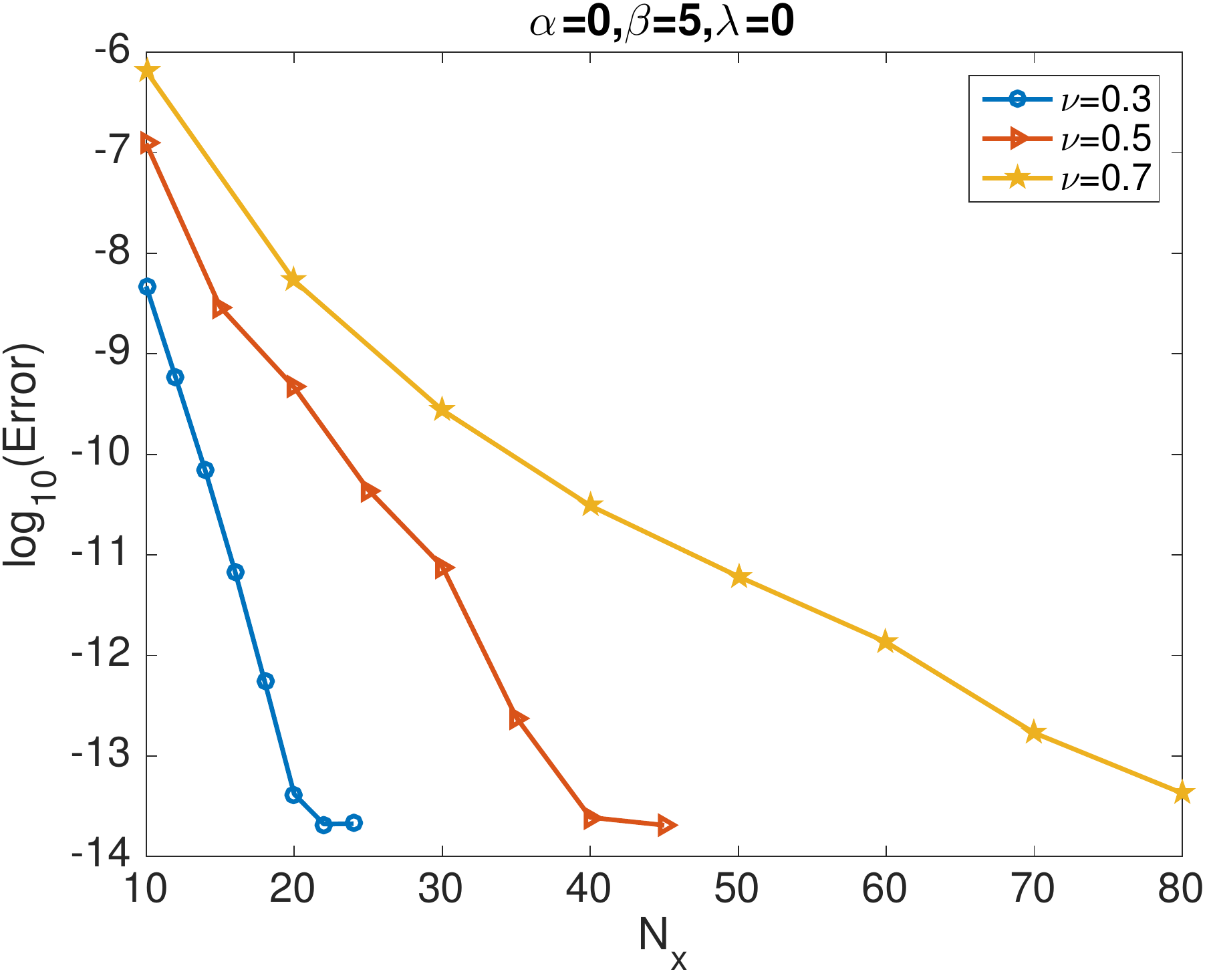}
\end{center}
\end{minipage}
\caption{  $f=e^{x_1x_2t},\quad T=\frac{1}{2}.$ }\label{graph_SM2}
\end{figure}

\section{Concluding remarks}
We constructed in  this paper two new classes of orthogonal functions, the log orthogonal functions (LOFs) and the generalized log orthogonal functions (GLOFs) by applying a log mapping to the Laguerre functions. We developed basic approximation theory for these new orthogonal functions. The approximate results reveal that  the  new orthogonal functions are particularly suitable for functions which have weak singularities at one endpoint. In particular, for functions involving one or multiple terms of $t^\alpha$ with $\alpha$ in an adjustable range, its approximation by the new orthogonal functions will converge exponentially, as opposed to a low algebraic rate if usual orthogonal polynomials are used.

As applications, we considered several typical fractional differential equations whose solutions usually exhibit weak singularities at one endpoint.
By using the GLOFs as basis functions, we constructed Galerkin methods for solving these fractional differential equations, and derived corresponding error analysis which reveals the fact that exponential convergence rate can be achieved even if the solution is weakly singular at one endpoint.
We provided ample numerical results to show that our methods based on  GLOFs are very effective for problems with solutions having weak singularities at one endpoint, such as the cases in many fractional differential equations. In particular, a special case of the GLOFs introduced in this paper has been used in \cite{CSZZ20} to develop a very efficient and accurate  spectral-Galerkin method (in the time direction) for
solving the time-fractional subdiffusion equations. 

The methods presented in this paper is limited to problems with singularities at one endpoint. To deal with problems having  singularities at both endpoints,  one can use  a two-domain  approach with  GLOFs on each subdomain, or to
 construct new classes of orthogonal functions which are suitable for problems having  singularities at both endpoints. This topic will be the subject of a future study.

\bigskip

\noindent{\bf Acknowledgment.} S. C. would like to thank Professor Lilian Wang for many useful suggestions and the discussion of the subsection \ref{sec3.3} during his visit at Nanyang Technological University, Singapore.

\begin{appendix}
\section{Some properties of Laguerre polynomials} \label{Appendix_Laguerre}
\renewcommand{\theequation}{L.\arabic{equation}}
\setcounter{equation}{0}
\noindent{\em The three-term recurrence}
\begin{equation}\begin{aligned}\label{Three-term R}
&\mathscr{L}^{(\alpha)}_0(y)=1,\qquad \mathscr{L}_1^{(\alpha)}=-y+\alpha+1,\\
&\mathscr{L}^{(\alpha)}_{n+1}(y)=\frac{2n+\alpha+1-y}{n+1}\mathscr{L}^{(\alpha)}_n(y)-\frac{n+\alpha}{n+1}\mathscr{L}^{(\alpha)}_{n-1}(y).
\end{aligned}\end{equation}
{\em Sturm-Liouville problem}
\begin{equation}\label{LaguSLProb}
 y^{-\alpha}e^y\partial_y\left( y^{\alpha+1}e^{-y}\partial_y \mathscr{L}_n^{(\alpha)}(y)\right)+n \mathscr{L}_n^{(\alpha)}(y)=0,
\end{equation}
{\em Derivative relations}
\begin{equation}
\mathscr{L}_n^{(\alpha)}(y)=\partial _y\mathscr{L}_{n}^{(\alpha)}(y)-\partial _y \mathscr{L}_{n+1}^{(\alpha)}(y),\label{LaguDerivative1}
\end{equation}
\begin{equation}
 y\partial _y \mathscr{L}_n^{(\alpha)}(y)=n\mathscr{L}_{n}^{(\alpha)}(y)-(n+\alpha)\mathscr{L}_{n-1}^{(\alpha)}(y),\label{LaguDerivative2}
\end{equation}
\begin{equation}
\partial _y \mathscr{L}_n^{(\alpha)}(y)=- \mathscr{L}_{n-1}^{(\alpha+1)}(y)=-\sum_{k=0}^{n-1}\mathscr{L}_{k}^{(\alpha)}(y).\label{LaguDerivative3}
\end{equation}
{\em Laguerre-Gauss quadrature}\\
Let $\{y_j^{(\alpha)}\}_{j=0}^N$ be the zeros of $\mathscr{L}_{n+1}^{(\alpha)}(y)$, then  the associated weights are
  \begin{equation}\label{Lagweights}
  \omega_j^{(\alpha)}= \frac{N+\alpha+1}{(N+\alpha+1)N!} \frac{y_j^{(\alpha)}}{[\mathscr{L}_{N}^{(\alpha)}(y_j^{(\alpha)})]^2}, \quad 0\leq j\leq N,
  \end{equation}
the { quadrature formula} is
  \begin{equation}\label{accurateintegral}
  \int_{\mathbb{R}^+} p(y) y^\alpha e^{-y} \,{\rm d}y= \sum_{j=0}^N p(y_j^{(\alpha)})\omega_j^{(\alpha)},\quad \forall \, p\in \mathcal{P}^y_{2N+1}.
  \end{equation}

\end{appendix}

\bibliography{ref_log}

\end{document}